\documentclass{amsart}
\usepackage{amsmath, amsthm, amssymb}
\usepackage{amsfonts}
\usepackage{verbatim}
\usepackage{color}
\numberwithin{equation}{section}

\newtheorem{theorem}{Theorem}
\newtheorem{lemma}{Lemma}

\theoremstyle{definition}
\newtheorem{remark}{Remark}
\newtheorem{example}{Example}

\newcommand{\cF}{\mathcal F}

\newcommand{\cP}{\mathcal P}
\newcommand{\cU}{\mathcal U}
\newcommand{\cV}{\mathcal V}
\newcommand{\cW}{\mathcal W}
\newcommand{\cT}{\mathcal T}
\newcommand{\cE}{\mathcal E}
\newcommand{\cD}{\mathcal D}

\newcommand{\cK}{\mathcal K}
\newcommand{\tD}{\text D}
\newcommand{\tN}{\text N}
\newcommand{\tR}{\text R}
\newcommand{\IR}{\mathbb R}

\newcommand{\IP}{\mathbb P}

\DeclareMathOperator{\Tr}{Tr}

\DeclareMathOperator{\meas}{meas}
\DeclareMathOperator{\diam}{diam}

\title[A posteriori estimates for cGM for an  
integro-differential eq.]
{\textbf{A posteriori error analysis for a continuous space-time 
finite element method for a hyperbolic integro-differential equation}}

\author[F.~Saedpanah]{Fardin Saedpanah}

\address{
Department of Mathematics, 
University of Kurdistan, P. O. Box 416, 
Sanandaj, Iran
}

\email{f.saedpanah@uok.ac.ir\\
           f\_saedpanah@yahoo.com}



\begin{document}

\begin{abstract}
An integro-differential equation of hyperbolic type, 
with mixed boundary conditions, is considered. 
A continuous space-time finite element method of degree one 
is formulated. 
A posteriori error representations based on space-time cells is
presented such that it can be used for adaptive strategies based
on dual weighted residual methods.
A posteriori error estimates based on weighted global 
projections and local projections are also proved.
\end{abstract}

\date{November 14, 2012}

\maketitle

\section{{\bf Introduction}}
We study the initial-boundary value problem  
(we use '$\cdot$' to denote the time derivative),
\begin{align}   \label{problem}
  \begin{aligned}
    &\ddot u(x,t)-\nabla\cdot\sigma(u;x,t)=f(x,t) 
      \quad && \textrm{in} \;\,\Omega\times (0,T),\\
    &u(x,t)=0 \quad&& \textrm{on}\;\Gamma_\text{D}\times (0,T),\\
    &\sigma(u;x,t)\cdot n=g(x,t)
      \quad&&\textrm{on}\;\Gamma_\tN\times (0,T),\\
    &u(x,0)=u^0(x),\quad\dot{u}(x,0)=v^0(x)
      \quad&&\textrm{in}\;\,\Omega,
  \end{aligned}
\end{align}
that is a hyperbolic type integro-differential equation arising, e.g., 
in modeling dynamic fractional/linear viscoelasticity. 
Here $u$ is the displacement vector, $f$ and $g$ represent, 
respectively, the volume and surface loads.
The stress $\sigma=\sigma(u;x,t)$ is determined by
\begin{equation*}
  \sigma(t)=\sigma_0(t)-\int_0^t\!\cK(t-s)\sigma_0(s)\,ds,
\end{equation*}
with
\begin{equation*}
  \sigma_0(t)=2\mu_0 \epsilon(t)+\lambda_0 \Tr (\epsilon(t))\text{I},
\end{equation*}
where $\text{I}$ is the identity operator,
$\epsilon$ is the strain which is defined by 
$\epsilon=\frac{1}{2}(\nabla u+(\nabla u)^T)$, and
$\mu_0,\lambda_0>0$ are elastic constants of Lam\'e type.
The kernel $\cK$ is considered to be either smooth (exponential), 
or no worse than weakly singular, so in both cases with the properties that
\begin{align}  \label{KernelProperty}
  \cK\geq 0,\quad \cK' \leq 0, \quad \|\cK\|_{L_1(\mathbb{R}^+)}=\kappa<1.
\end{align}
For examples of problems of this type see, e.g., \cite{StigFardin} 
and \cite{RiviereShawWhiteman2007} and their references. 

We note that, e.g., completely monotone functions,  
i.e., functions $b\in L_1(0,\infty)\cap \mathcal{C}^2(0,\infty)$, such that
\begin{equation*}
  (-1)^kD_t^k b(t) \ge 0,\quad t \in (0,\infty),\ k=0,1,2,
\end{equation*}
satisfy \eqref{KernelProperty}, 
when $\|b\|_{L_1(\mathbb{R}^+)} < 1$.
Hence, Mittag-Leffler type kernels, that are weakly singular and 
arise in fractional order viscoelasticity, 
satisfy \eqref{KernelProperty}, see, e.g., \cite{StigFardin}. 

Existence, uniqueness and regularity of solution of a problem
of the form \eqref{problem} has been studied in 
\cite{FardinarXiv:1203.4001}.  A posteriori analysis of
temporal finite element approximation of a parabolic type problem
and discontinuous Galerkin finite element approximation of a
quasi-static $(\ddot u\approx0)$ linear 
viscoelasticity problem has been studied, respectively, in
\cite{Stig4} and \cite{ShawWhiteman}. For analysis and numerical solution 
of integro-differential equations and related problems, 
from the extensive literature, we mention  
\cite{Stig4}, \cite{StigFardin}, \cite{McLeanThomee2010}, \cite{RiviereShawWhiteman2007}, 
and their references. 

Here we consider a continuous space-time finite element approximation 
of degree one, cG(1)cG(1), for problem \eqref{problem}. 
A similar method has been applied to the second order hyperbolic 
problems in \cite{BangerthGeigerRannacher}, where the main objective  
is goal-oriented adaptive discretization of the problems. 
We introduce a posteriori error representations, 
that can be a basis for goal-oriented adaptive strategies based on dual 
weighted residual (DWR) method, 
see \cite{BangerthGeigerRannacher} and \cite{BangerthRannacher:Book} 
for details and computational aspects on the DWR method. 
For examples of application of the DWR approach for problems in solid mechanics see 
\cite{KaramanouShawWarbyWhiteman2005}, 
\cite{ShawWarbyWhiteman2010}, and the references therein.

To evaluate the a posteriori error representations 
we need information about the continuous dual solution.
Such information has to be obtained either through a priori
analysis in form of bounds for the dual solution in certain 
Sobolev norms or through computation by solving the dual 
problem numerically. In this context we consider the former case 
and provide information through a priori analysis. 
To this end, we present two a posteriori error estimates based 
on global and local projections. 
A weighted global a posteriori error estimate is obtained, using 
global $L_2$-projections, for which the main framework is adapted from 
\cite{ErikssonEstepHansboJohnson}, and as an example of its specific application we refer to \cite{Johnson}. For the second a posteriori error 
estimate, that is based on local projections, we refer to 
\cite{BangerthGeigerRannacher}.  

The model problem \eqref{problem} is of hyperbolic type, and therefore 
our analysis are based on special treatment of discretization of the 
wave equation perturbed with a memory term. 
Here we only present the theory.

The memory term causes the growing amount 
of data that has to be stored and used in each time step. 
The most commonly used algorithms for this
integration are based on Lubich convolution quadrature \cite{Lubich} 
for fractional order
operators, see also \cite{LopezLubichSchadle} for an improved version. 
For examples of the application of this approch to overcome the problem 
with the growing amount of data, that has to be stored and used in time stepping 
methods, see \cite{Stig2}, \cite{Stig3}, \cite{Stig4}, \cite{LubichSloanThomee}, 
and \cite{McLeanThomeeWahlbin}. 
In particular see, e.g., \cite{Stig2} where sparse quadrature together with 
appropriate a posteriori error representaions and estimates have been 
used in adaptive strategies based on the DWR approach. 
We plan to address this issue together with numerical adaptation 
methods such as DWR approach in future work.  
However, we should note that this is not an issue for exponentially decaying 
memory kernels, in linear viscoelasticity, that are represented as a Prony series.  
In this case recurrence relationships can be derived which means recurrence formula 
are used for history updating, see \cite{KaramanouShawWarbyWhiteman2005} 
and \cite{ShawWhiteman} for more details. 

The outline of this paper is as follows. In $\S2$ we define
a weak form of \eqref{problem} and the corresponding
dual (adjoint) problem. 
In $\S3$ we formulate a continuous Galerkin method of degree 
one. 
Then in $\S4$ we present a posteriori error representations, from which,  
in $\S5$ we prove a weighted global a posteriori error estimate. 
A  localized a posteriori error estimate is also proved in $\S6$. 
We remark on special cases regarding the memory term and 
regularity of the convolution kernel, when it is needed.

\section{{\bf Weak formulation and stability}}
We let $\Omega \subset\mathbb{R}^d, \, d=2,3$, be a bounded polygonal domain
with boundary $\Gamma=\Gamma_\tD\cup\Gamma_\tN$ where
$\Gamma_\tD$ and $\Gamma_\tN$ are
disjoint and $\meas(\Gamma_\tD)\not=0$.
We introduce the function spaces
$H=L_2(\Omega)^d,\,H_{\Gamma_\tN}=L_2(\Gamma_\tN)^d,\,$
and
$V=\{v\in H^1(\Omega)^d:v\!\!\mid_{\Gamma_\tD}=0\}$.
We denote the norms in $H$ and $H_{\Gamma_\tN}$ by
$\|\cdot\|$ and $\|\cdot\|_{\Gamma_\tN}$, respectively. 
\subsection{Weak formulation}
We define a bilinear form (with the usual summation convention)
\begin{equation}\label{bilinear}
 a(v,w)=\int_{\Omega}\!\big(2\mu_0\epsilon_{ij}(v)\epsilon_{ij}(w) +
\lambda_0\epsilon_{ii}(v)\epsilon_{jj}(w)\big)\,dx,\quad v,w\in V,
\end{equation}
which is coercive on $V$, 
and we equip $V$ with the inner product
$a(\cdot,\cdot)$ and norm
$\|v\|_V^2=a(v,v)$. 
We define $Au=-\nabla\cdot\sigma_0(u)$, which is a selfadjoint, positive
definite, unbounded linear operator, with
$\cD(A)=H^2(\Omega)^d\cap V$, and we use the norms 
$\|v\|_s=\|A^{s/2}v\|$.

We use a ``velocity-displacement'' formulation of \eqref{problem}
which is obtained by introducing a new velocity variable.
Henceforth we use the new variables $u_1=u,\,u_2=\dot u$ and $u=(u_1,u_2)$
the pair of vector valued functions. 
Then the variational form is to find 
$u_1(t),\,u_2(t)\in V$ such that
$u_1(0)=u^0$, $u_2(0)=v^0$, and
\begin{equation} \label{weakform}
  \begin{split}
    &(\dot u_1(t),v_1)-(u_2(t),v_1)=0,\\
    &(\dot u_2(t),v_2) + a(u_1(t),v_2)
    - \int_0^t \cK(t-s) a(u_1(s), v_2) \, ds \\
    &\qquad\qquad\qquad\quad\quad
           = (f(t),v_2) + (g(t),v_2)_{\Gamma_N},
    \quad \forall v_1,v_2 \in V ,\,t\in (0,T).
  \end{split}
\end{equation}

Now we define the bilinear and linear forms
$B:\cU\times \cV \to \mathbb{R}$ and 
$L:\cV \to \mathbb{R}$  
by
\begin{align}  \label{BL}
  \begin{aligned}
    B(u,v)
    &=\int_0^T\!
      \Big\{(\dot u_2,v_2)+a(u_1,v_2)
        -\int_0^t\!\cK(t-s)a\big(u_1(s),v_2\big)\,ds\\
    &\quad +(\dot u_1,v_1)-(u_2,v_1)\Big\}\,dt
      +\big(u_1(0),v_1(0)\big)+\big(u_2(0),v_2(0)\big),\\
    L(v)&=\int_0^T\!\Big\{(f,v_2)+(g,v_2)_{\Gamma_\tN}\Big\}\,dt
      +\big(u^0,v_1(0)\big)+\big(v^0,v_2(0)\big),
  \end{aligned}
\end{align}
where
\begin{align}   \label{UV}
  \begin{aligned}
    \cU&=H^1(0,T;V)\times H^1(0,T;H),\\
    \cV&=\big\{v=(v_1,v_2):
      v\in L_2(0,T;H)\times L_2(0,T;V),
      v_i \text{ right continuous in }t \big\}.
  \end{aligned}
\end{align}
We note that if we change the order of the time integrals in the 
convolution term as well as changing the role of the variables $s,t$, 
we have the second variant of the bilinear form $B$, that is, 
\begin{align}  \label{B2}
  \begin{aligned}
    B(u,v)
    &=\int_0^T\!
      \Big\{(\dot u_2,v_2)+a(u_1,v_2)
        -\int_t^T\!\cK(s-t) a\big(u_1(t),v_2(s)\big)\,ds \\
    &\quad +(\dot u_1,v_1)-(u_2,v_1) \Big\}\,dt
      +\big(u_1(0),v_1(0)\big)+\big(u_2(0),v_2(0)\big).
  \end{aligned}
\end{align}
We use this variant of $B$ in $\S4$ for a posteriori error analysis.

The weak form \eqref{weakform} can be writen as: 
find $u\in \cU$ such that,
\begin{equation} \label{weakformprimary}
  B(u,v)=L(v),\quad \forall v\in\cV.
\end{equation}
Here the definition of the velocity $u_2=\dot u_1$ is enforced in the
$L_2$ sense, and the initial data are placed in the bilinear form in
a weak sense. A variant is used in \cite{StigFardin} where the
velocity has been enforced in the $H^1$ sense, without placing the initial data in the bilinear form. 
We also note that the initial data  are retained by the choice of 
the function space $\cV$, that
consists of right continuous functions with respect to time.

Our error analysis is based on the duality arguments, 
and therefore we formulate 
the dual form of \eqref{weakformprimary}. 
To this end, we define the bilinear and linear forms
$B_\tau^*:\cV^*\times\cU^*\to\IR,\,
L_\tau^*:\cV^*\to\IR$,
 for $\tau\in \mathbb{R}^{\ge 0}$, by
\begin{align}  \label{BLdual}
  \begin{aligned}
    B_\tau^*(v,z)
    &=\int_{\tau}^T\!
      \Big\{-(v_1,\dot z_1)+a(v_1,z_2)
        -\int_t^T\!\cK(s-t)a\big(v_1,z_2(s)\big)\,ds\\
    &\qquad - (v_2,\dot{z}_2)-(v_2,z_1) \Big\}\,dt
        +\big(v_1(T),z_1(T)\big)
        +\big(v_2(T),z_2(T)\big),\\
    L^*_\tau(v)
    &=\int_{\tau}^{T}\!
      \Big\{(v_1,j_1)+(v_2,j_2)\Big\}\,dt
        +\big(v_1(T),z_1^{T}\big)
        +(v_2(T),z_2^{T}),
  \end{aligned}
\end{align}
where $j_1,\,j_2$ and $z_1^{T},\,z_2^{T}$ represent, respectively,
the load terms and the initial data of the dual (adjoint) problem.
In case of $\tau=0$, we use the notation
$B^*,L^*$ for short.
Here
\begin{align}   \label{UVdual}
  \begin{aligned}
    \cU^*&=H^1(0,T;H)\times H^1(0,T;V),\\
    \cV^*&=\big\{v=(v_1,v_2):
      v\in L_2(0,T;V)\times L_2(0,T;H),
      v_i \text{ left continuous in }t \big\}.
  \end{aligned}
\end{align}
We note that, recalling \eqref{UV},  
$\cU\subset \cV^*,\,\cU^*\subset\cV$.

We also note that $B^*$ is the adjoint form of $B$. 
Indeed, integrating by parts with respect to time 
in $B$, then changing the order of integrals in the convolution 
term as well as changing the role of the variables $s,t$, we have,
\begin{equation}   \label{BequalBstar}
  B(u,v)=B^*(u,v),\quad \forall u\in\cU,\, v\in\cU^*.
\end{equation}

Hence, the variational form of the dual problem
is to find $z\in \cU^*$ such that,
\begin{equation}   \label{weakformdual}
  B^*(v,z)=L^*(v),\quad \forall v\in\cV^*,
\end{equation}
that is a weak formulation of
\begin{equation}   \label{z2dual}
  \ddot z_2+Az_2-\int_t^T\!
  \cK(s-t)Az_2(s)\,ds=j_1-\frac{\partial}{\partial t} j_2,
\end{equation}
with initial data $z_1^T,z_2^T$, and function $j=(j_1,j_2)$ that 
is defined by $L^*(w)=\int_0^T\!(w,j)\,dt$.

We note that, since $z$ is a solution of \eqref{weakformdual}, 
we have 
\begin{equation}   \label{z1z2}
  z_1=-\dot z_2-j_2.
\end{equation}
\subsection{Stability of the solution of the dual problem}
We know that stability estimates and the corresponding anlysis for 
dual problem is similar to the primal problem, however with opposite 
time direction. 
Hence, having a smooth or weakly singular kernel with 
\eqref{KernelProperty}, we can obtain the stability estimates 
\eqref{continuousstabilitydualineq} from, e.g., \cite{Stig1} or \cite{StigFardin} 
for the continuous dual solution. Here we state stability estimates for 
the continuous dual solution in the following lemma, 
and we omit the proof for short. 
We note that the stability constant in \eqref{continuousstabilitydualineq} 
does not depend on $t$, and Gronwall's lemma has not been used, 
see \cite{Stig1} and \cite{StigFardin}.  
See also \cite{RiveraMenzala1999}, \cite{RiviereShawWhiteman2007} and 
\cite{ShawWhiteman},  where stability estimates have been represented, 
in which the stability factor depends on $t$, due to Gronwall's lemma. 
\begin{lemma}
Let $z$ be the solution of the dual problem
\eqref{weakformdual} with sufficiently smooth data
$z_1^T,z_2^T, j_1,j_2$. Then, for some constant $C=C(\kappa)$, we have stability estimates
\begin{equation}   \label{continuousstabilitydualineq}
  \begin{split}
    \|z_1(t)\|_l+\|z_2(t)\|_{l+1}
    &\le C\Big\{\|z_1^T\|_l+\|z_2^T\|_{l+1}
    +\int_t^T\!\Big(\|j_1\|_l+\|j_2\|_{l+1} \Big)\,dr\Big\}.
  \end{split}
\end{equation}
\end{lemma}
We note that, if we set $j_1=j_2=0$ in \eqref{continuousstabilitydualineq} 
and recall \eqref{z2dual}--\eqref{z1z2}, 
then for $r\in[t,T]$ and $l\in \mathbb{R}$, we have
\begin{equation*}
  \begin{split}
    \|\dot z_1(r)\|_{l-1}
    &=\|\ddot z_2(r)\|_{l-1}
    =\big\| -Az_2(r)+\int_r^T\! \cK(s-r)Az_2(s)\ ds\big\|_{l-1}\\
    &\le \| Az_2(r)\|_{l-1} +\int_r^T\! \cK(s-r) \| Az_2(s)\|_{l-1}\ ds,
  \end{split}
\end{equation*}
that, recalling $\|\cK\|_{L_1(\mathbb{R})}=\kappa$, for some $C=C(\kappa)$ implies
\begin{equation*}
  \|\dot z_1(r)\|_{l-1}
  =\|\ddot z_2(r)\|_{l-1}\\
  \le (1+\kappa) \max_{r \leq s\leq T} \| Az_2(r)\|_{l-1} 
  \le (1+\kappa) \max_{r \leq s\leq T} \| z_2(r)\|_{l+1} .
\end{equation*}
From this and \eqref{continuousstabilitydualineq} we conclude the stability inequality 
\begin{equation}   \label{stabilitydualineq}
  \begin{split}
    \|\dot z_1(t)\|_{l-1}+\|z_1(t)\|_l+\|z_2(t)\|_{l+1}
    &\le C\big\{\|z_1^T\|_l+\|z_2^T\|_{l+1}\big\}.
  \end{split}
\end{equation}

\section{{\bf The continuous Galerkin method}}
Here we formulate the continuous Galerkin method of order one, 
cG(1)cG(1), for the primary and dual problems \eqref{weakformprimary} 
and \eqref{weakformdual}. 
\subsection{The cG method}
Let $0=t_0<t_1<\cdots<t_{n-1}<t_n<\cdots<t_N=T$ be a partition of the time interval $[0,T]$.
To each discrete time level $t_n$ we associate a triangulation $\cT_h^n$ of the polygonal domain $\Omega$ with the mesh function,
\begin{equation}   \label{hn}
  h_n(x)=h_K=\diam(K),\quad x\in K,\,K\in\cT_h^n,
\end{equation}
and a finite element space $V_h^n$ consisting of continuous
piecewise linear polynomials.
For each time subinterval $I_n=(t_{n-1},t_n)$ of length $k_n=t_n-t_{n-1}$,
we define intermediate triangulaion $\bar \cT_h^n$ which is composed
of mutually finest meshes of the neighboring meshes
$\cT_h^n,\,\cT_h^{n-1}$ defined at discrete time levels $t_n,\,t_{n-1}$, respectively.
The mesh function $\bar h_n$ is then defined by
\begin{equation}   \label{barhn}
  \bar h_n(x)=\bar h_K=\diam(K),\quad x\in K,\,K\in\bar\cT_h^n.
\end{equation}
Correspondingly, we define the finite element spaces
$\bar V_h^n$ consisting of continuous piecewise linear polynomials.
This construction is used in order to allow continuity in time of the
trial functions when the meshes change with time.
Hence we obtain a decomposition of each time slab
$\Omega^n=\Omega\times I_n$ into space-time cells
$K^n=K\times I_n,\,K\in \bar\cT_h^n$
(prisms, for example, in case of $\Omega\subset \mathbb{R}^2$).
We note the difference between the mesh functions $h_n$ and
$\bar h_n$, and this is important in our a posteriori error
analysis.
The trial and test function spaces for the discrete form are,
respectively:
\begin{equation}   \label{discretespaces}
  \begin{split}
    \cU_{hk}&=\Big\{U=(U_1,U_2):
      U \text{ continuous in } \Omega\times [0,T],
        U(x,t)|_{I_n} \text{ linear in } t,\\
      &\qquad\qquad\qquad\qquad\
        U(\cdot,t_n)\in (V_h^n)^2, U(\cdot,t)|_{I_n}
          \in (\bar V_h^n)^2 \Big\},\\
    \cV_{hk}&=\Big\{V=(V_1,V_2):
      V(\cdot,t) \text{ continuous in } \Omega,
         V(\cdot,t)|_{I_n}\in (V_h^n)^2,\\
       &\qquad\qquad\qquad\qquad\
         V(x,t)|_{I_n} \text{ piecewise constant in } t\Big\}.
  \end{split}
\end{equation}
We note that global continuity of the trial functions in $\cU_{hk}$ 
requires the use of `\textit{hanging nodes}' if the spatial mesh changes 
across a time level $t_n$. 
In these `irregular' nodal points the unknowns are 
eliminated by interpolating values at neighboring `regular' nodal points, 
see \cite{BangerthGeigerRannacher}, \cite{Svensson2006} 
and the references therein for practical aspects. 
We allow one hanging node per edge or face.
On the other hand, 
the test functions, which are allowed to be discontinuous in time, 
are defined in the time slabs $\Omega\times I_n$ on the spatial meshes corresponding 
to $t_n$. 
In this discretization trial functions have
support in both time intervals $I_n$ and $I_{n+1}$ adjacent to time instants $t_n$. 
Therefore on each time interval $I_n$, the trial functions defined at time 
instant $t_{n-1}$ on a mesh $\cT_h^{n-1}$ overlap with test functions defined 
at time instant $t_n$ on a mesh $\cT_h^n$.  
Hence, 
to form space-time integrals of multiplication of trial functions and test functions,  
we need the `union' of the two triangulations, that we have denoted by $\bar\cT_h^n$. 
 We also note that, due to the structure of the elements of 
 $\bar\cT_h^n$, computations are feasible with reasonable effort if the grids 
 $\cT_h^{n-1}$ and  $\cT_h^n$ are related. Hierarchically structured spatial 
meshes, where  $\bar\cT_h^n$ is the set of most refined cells from the two grids, 
reduces the substantial work caused by the mesh transfer  from 
one time level to the next. 
More details and references on the practical implementation  
can be found in \cite{BangerthGeigerRannacher}.

In the construction of $\cU_{hk}$ and $\cV_{hk}$ we have associated
the triangulation $\cT_h^n$ with discrete time levels instead of the time
slabs $\Omega^n$, and in the interior of time slabs we let $U$ be from the
union of the finite element spaces defined on the triangulations at the
two adjacent time levels.
Associating triangulation with time slabs instead of time levels would
yield a variant scheme which includes jump terms due to discontinuity
at discrete time leveles, when coarsening happens.
This means that there are extra degrees of freedom that one might use
suitable projections for transfering solution at the time levels $t_n$,
 see \cite{StigFardin}.

The continuous Galerkin method, based on the variational formulation \eqref{weakform}, is to find $U\in\cU_{hk}$ such that,
\begin{equation}   \label{cG}
  B(U,V)=L(V),\quad \forall\, V\in \cV_{hk}.
\end{equation}
The Galerkin orthogonality, with $u=(u_1,u_2)$ being the exact solution
of \eqref{weakform}, is then,
\begin{equation}   \label{Galerkinorthogonality}
  B(U-u,V)=0,\quad \forall\, V\in \cV_{hk}.
\end{equation}
From \eqref{cG} we can recover the time stepping scheme,
\begin{equation}   \label{cG2}
  \begin{split}
    &\int_{I_n} \!\Big\{(\dot{U}_1,V_1)-(U_2,V_1) \Big\}\,dt=0,\\
    &\int_{I_n}\!\Big\{(\dot{U}_2,V_2)+a(U_1,V_2)
       -\int_0^t \! \cK(t-s)a\big(U_1(s),V_2(t)\big)\,ds\Big\}\,dt\\
    &\qquad\qquad
      =\int_{I_n}\!
      \big\{(f,V_2)\,dt+(g,V_2)_{\Gamma_\tN}\big\}\, dt ,
      \quad\forall\, (V_1,V_2)\in \cV_{hk},\\
    &U_1(0)=u_h^0,\quad U_2(0)=v_h^0,
  \end{split}
\end{equation}
for suitable choice of $u_h^0,v_h^0 \in V_h^0$ as approximations of
the initial data $u^0,v^0$. 

We define the orthogonal projections 
$\cP_{h,n}:H\to V_h^n$ and $\cP_{k,n}:L_2(I_n)^d\to\IP^d_0(I_n)$,
respectively, by
\begin{align}  \label{projections}
  \begin{aligned}
    (\cP_{h,n}v-v,\chi)&=0,&&\forall v\in H,\,
      \chi\in V_h^n,\\
    \int_{I_n}\!\!(\cP_{k,n}v-v)\cdot \psi\, dt&=0,&&
       \forall  v\in L_2(I_n)^d, \,\psi\in \IP^d_0(I_n)\,,
  \end{aligned}
\end{align}
with $\IP_0^d$ denoting the set of all vector-valued constant
polynomials. Correspondingly, we define $\cP_hv$
and $\cP_kv$ for
$t\!\in\! I_n\,(n=1,\cdots,N)$, by
$(\cP_hv)(t)=\cP_{h,n}v(t)$ and
$\cP_kv=\cP_{k,n}(v\!\!\mid_{I_n})$.
We note that, as a natural choice, we can set the initial data in \eqref{cG2} as
\begin{equation}   \label{initialdata}
  u_h^0=\cP_h u^0,\quad v_h^0=\cP_h v^0.
\end{equation}

We also note that, when we do not change the spatial mesh or just refine the
spatial mesh from one time level to the next one, i.e., 
$  V_h^{n-1}\subset V_h^n,\ n=1,\dots,N$, 
then we have $\bar V_h^n=V_h^n$. 

We introduce the linear operator
$A_{n,r}:V_h^r\to V_h^n$ by
\begin{equation*}
  a(v_r,w_n)
   =(A_{n,r}v_r,w_n),\quad\forall v_r\in V_h^r,
    \,w_n \in V_h^n.
\end{equation*}
We set $A_n=A_{n,n}$, with discrete norms
\begin{equation*}
  \| v_n\|_{h,l}=\| A_{n}^{l/2} v_n\|
     =\sqrt{(v_n,A_n^l v_n)},\,\quad v_n\in V_h^n\,\,
       \textrm{and}\,\, l\in \mathbb{R}\,,
\end{equation*}
and $A_h$ so that $A_h v=A_n v$ for $v \in V_h^n$.
We use $\bar A_h$ when it acts on $\bar V_h^n$.
\section{{\bf A posteriori error estimation: error representation}}
Having certain regularity on the data, i.e., initial data $u^0,v^0$
and the force terms $f,g$, there are still two types of limitation
for higher global regularity of a weak solution of \eqref{problem}.
One is due to the mixed Dirichlet-Neumann boundary condition.
This type of boundary condition are natural in practice, and a pure
Dirichlet boundary condition cannot be realistic in applications.
Other limitation is the singularity of the convolution kernel $\cK$.
This means that even with the pure Dirichlet boundary condition, higher
regularity of a weak solution is limited, 
see \cite{StigFardin}, \cite{FardinarXiv:1203.4001},
though with smoother kernels we can get higher regularity.
These, and other general motivations such as no practical use of
a priori error estimates, call for adaptive meshes based on
a posteriori error analysis.

Here a space-time cellwise error representation is given.
The main framework is adapted from \cite{BangerthRannacher:Book}, and
a general linear goal functional $L^*(\cdot)$ is used. 
For an example of a global goal functional, see Example \ref{Example}. 
This error representation can be used for goal-oriented
adaptive strategies based on dual weighted residual (DWR) method. 
For more details on dual weighted residual method and
its practical aspects for differential equations,
see \cite{BangerthGeigerRannacher}, 
\cite{BangerthRannacher:Book} and references therein.

In the following, we recall and define some notations for this section. 
We denote the space-time cells $K^n=K\times I_n$ 
and $\partial K^n=\partial K\times I_n$, for $K \in \bar\cT_h^n$, 
and it should be noticed that $\partial K^n$ is not the boundary 
of $K^n$. 
Also for a simplex $K$, we define the inner products 
$(\cdot,\cdot)_{K^n}=\int_{I_n}(\cdot,\cdot)_{K}\  dt$,  
$(\cdot,\cdot)_{\partial K^n}=\int_{I_n}(\cdot,\cdot)_{\partial K}\  dt$,
and the corresponding $L_2$-norms 
\begin{equation*}
  \begin{split}
    &\| \cdot\|_{K^n}=\Big( \int_{I_n} \|\cdot\|_{K}^2\  dt \Big)^{1/2}
      =\Big( \int_{I_n} \|\cdot\|_{L_2(K)}^2\  dt \Big)^{1/2}, \\
    &\| \cdot\|_{\partial K^n}
      =\Big( \int_{I_n} \|\cdot\|_{\partial K}^2\  dt \Big)^{1/2}
      =\Big( \int_{I_n} \|\cdot\|_{L_2(\partial K)}^2\  dt \Big)^{1/2}.
  \end{split}
\end{equation*}
We also denote
\begin{equation}   \label{convolution}
  (\cK*v)^j(t)=\int_{t_{j-1}}^{t_j\wedge t}\cK(t-s)v(s)\ ds. 
\end{equation}
Throughout we use the usual notations 
$a\wedge b = \min\{a,b\}$ and $a\vee b =\max\{a,b\}$. 

Now we present three a posteriori error representations, 
and we note that the second error representation, \eqref{errorrep2}, 
is space-time cellwise.


\begin{theorem}
Let $u$ and $U$ be the solutions of \eqref{weakformprimary} 
and \eqref{cG}, respectively, 
and $L^*(\cdot)$ be the linear functional defined in \eqref{BLdual}.
Then, denoting the error $e=U-u$, we have the error representations
\begin{equation}   \label{errorrep1}
  L^*(e)=\sum_{K\in\cT_h^0}\Theta_{0,K}
     +\sum_{n=1}^N \sum_{K\in\bar\cT_h^n}
     \sum_{i=1}^4 \Theta_{i,K}^n
     +\sum_{n=1}^N \sum_{j=1}^n\sum_{K\in\bar\cT_h^j}\Theta_{5,K}^{n,j},\\
\end{equation}
\begin{equation}   \label{errorrep2}
  L^*(e)=\sum_{K\in\cT_h^0}\Theta_{0,K}
     +\sum_{n=1}^N \sum_{K\in\bar\cT_h^n}
     \sum_{i=1}^5 \Theta_{i,K}^n,\\
\end{equation}
\begin{equation}  \label{errorrep3}
  L^*(e)=\sum_{K\in\cT_h^0}\Theta_{0,K}
     +\sum_{n=1}^N \sum_{K\in\bar\cT_h^n}
     \sum_{i=1}^4 \Theta_{i,K}^n
     +\sum_{n=1}^N \sum_{j=n}^N\sum_{K\in\bar\cT_h^j}\Theta_{5,K}^{N,j},
\end{equation}
where, with $z_{hk}\in\cV_{hk}$ being an approximation of the dual
solution $z$ and $E_{hk}z=z_{hk}-z$ being the error operator,
\begin{equation}   \label{thetas}
  \begin{split}
    &\Theta_{0,K}=
     \big(U_1(0)-u^0,E_{hk}z_1(0)\big)_{K}
     +\big(U_2(0)-v^0,E_{hk}z_2(0)\big)_{K},\\
    &\Theta_{1,K}^n=
     (\dot U_1-U_2,E_{hk}z_1)_{K^n},\ \quad
     \Theta_{2,K}^n=
      (\dot U_2-f,E_{hk}z_2)_{K^n},\\
    & \Theta_{3,K}^n=
      (g_d-g,E_{hk}z_2)_{\partial K^n}, \qquad 
      \Theta_{4,K}^n=
      (r_d,E_{hk}z_2)_{\partial K^n},\\
    &  \Theta_{5,K}^{n,j}=
      -\int_{I_n}\Big((\cK*r_d)^j(t),E_{hk}z_2(t)\Big)_{\partial K}\ dt, \\
    &\Theta_{5,K}^n=
      -\Big(r_d,\int_t^T\! \cK(s-t)E_{hk}z_2(s)\,ds\Big)_{\partial K^n}\\
    &  \Theta_{5,K}^{N,j}=
      -\int_{I_n}\Big(r_d(t),
        \int_{t_{j-1}\vee t}^{t_j}\cK(s-t)E_{hk}z_2(s)\ ds\Big)_{\partial K}\ dt.
  \end{split}
\end{equation}
Here $r_d$ are the residuals representing the jumps of the normal
derivatives $\sigma_0(U_1)\cdot n$, which are determined by
\begin{equation}   \label{rh}
  r_d|_\Gamma =
  \begin{cases}
    -\frac{1}{2}[\sigma_0(U_1)\cdot n]
     &\text{if }\ \Gamma\subset\partial K
       \setminus\partial\Omega,\\
    0&\text{if }\ \Gamma\subset \partial\Omega,
  \end{cases}
\end{equation}
and $g_d$ are the contribution from the Neumann boundary condition 
defined by
\begin{equation}   \label{gh}
  g_d|_\Gamma =
  \begin{cases}
    \sigma(U_1)\cdot n
      =\Big(\sigma_0(U_1)-\int_0^t\!\cK(t-s)\sigma_0(U_1(s))\ ds\Big)\cdot n
      &\text{if }\ \Gamma\subset\partial K\cap\Gamma_{\tN},\\
    0&\text{otherwise}.
  \end{cases}
\end{equation}
\end{theorem}


\begin{proof}
Using the identity \eqref{BequalBstar} and the Galerkin
orthogonality \eqref{Galerkinorthogonality} we have,
\begin{equation}   \label{errorrep:eq0}
  \begin{split}
    L^*(e)&=B^*(e,z)=B(e,z)
      =B(e,E_{hk}z)
      =B(U,E_{hk}z)-B(u,E_{hk}z)\\
    &=B(U,E_{hk}z)-L(E_{hk}z)=R(U;E_{hk}z),
  \end{split}
\end{equation}
where $R(U;\cdot)$ is the residual of the Galerkin
approximation $U$ as a functional on the solution space $\cU^*$. 
Then, by the definition of $B$ and $L$ in \eqref{BL}, we have
\begin{equation}   \label{errorrep:eq1}
  \begin{split}
    L^*(e)&=
    \big(U_1(0)-u^0,E_{hk}z_1(0)\big)
    +\big(U_2(0)-v^0,E_{hk}z_2(0)\big)\\
    &\quad+\int_0^T\! \Big\{
     (\dot U_1,E_{hk}z_1)-(U_2,E_{hk}z_1)
      +(\dot{U}_2,E_{hk}z_2)+a(U_1,E_{hk}z_2)\\
    &\quad\quad-\int_0^t\!
      \cK(t-s)a\big(U_1(s),E_{hk}z_2(t)\big)\,ds\Big\}\,dt\\
    &\quad-\int_0^T\!
      \Big\{(f,E_{hk}z_2)+(g,E_{hk}z_2)_{\Gamma_\tN}\Big\}\,dt.
  \end{split}
\end{equation}
Now, by partial integration with respect to the space variable and 
recalling $r_d$ from \eqref{rh}, we obtain
\begin{equation}   \label{errorrep:intbyparts1}
  \begin{split}
     &\int_0^T\! a(U_1,E_{hk}z_2)\,dt\\
     &\quad =\sum_{n=1}^N\int_{I_n}
      \sum_{K\in\bar\cT_h^n}a(U_1,E_{hk}z_2)_{K}\,dt
    =\sum_{n=1}^N\int_{I_n}
      \sum_{K\in\bar\cT_h^n}
      \big(\sigma_0(U_1)\cdot n,E_{hk}z_2\big)_{\partial K}\,dt\\
    &\quad=\sum_{n=1}^N\int_{I_n}\Big\{
      \sum_{E\in \cE_I^n}
      \big(-[\sigma_0(U_1)\cdot n],E_{hk}z_2\big)_E
       +\sum_{E\in \cE_{\Gamma_\tN}^n}
        \big(\sigma_0(U_1)\cdot n,E_{hk}z_2\big)_E\Big\}\,dt\\
    &\quad=\sum_{n=1}^N \int_{I_n}\sum_{K\in \bar\cT_h^n}
      \Big\{
      (r_d,E_{hk}z_2)_{\partial K}
      +\big(\sigma_0(U_1)\cdot n,E_{hk}z_2\big)_{\partial K \cap \Gamma_\tN}
      \Big\}\,dt,
  \end{split}
\end{equation}
where $\cE_I^n,\,\cE_{\Gamma_\tN}^n$ are, respectively, the sets of the
interior edges and the edges on the Neumann boundary, corresponding to the triangulation $\bar\cT_h^n$.

For the convolution term in \eqref{errorrep:eq1}, 
similar to \eqref{errorrep:intbyparts1}, we have
\begin{equation*}   
  \begin{split}
    \int_0^T\int_0^t\! &\cK(t-s)a(U_1(s),E_{hk}z_2(t))\ ds\ dt\\
    &=\sum_{n=1}^N\int_{I_n} \sum_{j=1}^n\int_{t_{j-1}}^{t_j\wedge t}
       \cK(t-s)\sum_{K\in\bar\cT_h^j} a\big(U_1(s),E_{hk}z_2(t)\big)_K\ ds\  dt\\
    &=\sum_{n=1}^N\int_{I_n} \sum_{j=1}^n
       \int_{t_{j-1}}^{t_j\wedge t}\cK(t-s) 
      \bigg\{ \sum_{E\in \cE_I^j}
      \Big(-[\sigma_0(U_1)\cdot n],E_{hk}z_2(t) 
      \Big)_E\\
    &\quad+\sum_{E\in \cE_{\Gamma_\tN}^j}
      \Big(\sigma_0(U_1)\cdot n,E_{hk}z_2(t) 
      \Big)_E\bigg\}\ ds\\
    &=\sum_{n=1}^N\int_{I_n} \sum_{j=1}^n
       \int_{t_{j-1}}^{t_j\wedge t}\cK(t-s) \sum_{K\in\bar\cT_h^j}
       \Big\{
       \Big(r_d(s),E_{hk}z_2(t)\Big)_{\partial K}\\
    &\quad  + \Big(\sigma_0(U_1(s))\cdot n,E_{hk}z_2(t) 
      \Big)_{\partial K \cap \Gamma_{\tN}}
       \Big\}\ ds\ dt,
  \end{split}
\end{equation*}
that implies
\begin{equation*}   
  \begin{split}
    \int_0^T\int_0^t\! &\cK(t-s)a(U_1(s),E_{hk}z_2(t))\ ds\ dt\\
    &=\sum_{n=1}^N\int_{I_n} \sum_{j=1}^n
       \sum_{K\in\bar\cT_h^j}
       \Big\{
       \Big(\int_{t_{j-1}}^{t_j\wedge t}\cK(t-s) r_d(s)\ ds,
        E_{hk}z_2(t)\Big)_{\partial K}\\
    &\quad  
      +\Big(\int_{t_{j-1}}^{t_j\wedge t}\cK(t-s) \sigma_0(U_1(s))\cdot n\ ds,
        E_{hk}z_2(t) 
      \Big)_{\partial K \cap \Gamma_{\tN}}
       \Big\}\ dt\\
    &=\sum_{n=1}^N \sum_{j=1}^n\sum_{K\in\bar\cT_h^j}
       \int_{I_n}\Big(\int_{t_{j-1}}^{t_j\wedge t}\cK(t-s)r_d(s)\ ds,
         E_{hk}z_2(t)\Big)_{\partial K}dt\\
    &\quad +\sum_{n=1}^N\int_{I_n}
       \Big(\int_0^t \cK(t-s) \sigma_0(U_1(s))\cdot n\ ds,E_{hk}z_2(t) 
      \Big)_{\Gamma_{\tN}}\ dt\\
    &=\sum_{n=1}^N \sum_{j=1}^n\sum_{K\in\bar\cT_h^j}
       \int_{I_n}\Big(\int_{t_{j-1}}^{t_j\wedge t}\cK(t-s)r_d(s)\ ds,
         E_{hk}z_2(t)\Big)_{\partial K}dt\\
    &\quad +\sum_{n=1}^N\int_{I_n}\sum_{K\in\bar\cT_h^n}
       \Big(\int_0^t \cK(t-s) \sigma_0(U_1(s))\cdot n\ ds,E_{hk}z_2(t) 
      \Big)_{\partial K \cap \Gamma_{\tN}}\ dt.
  \end{split}
\end{equation*}
This, together with \eqref{errorrep:intbyparts1} and 
recalling $g_d$ from \eqref{gh}, imply
\begin{equation}   \label{errorrep:eq2}
  \begin{split}
    \int_0^T\! 
    &\Big\{a(U_1(s),E_{hk}z_2(t)) 
       -\int_0^t\! \cK(t-s)a(U_1(s),E_{hk}z_2(t))\ ds
       \Big\}\ dt\\
    &=\sum_{n=1}^N \sum_{K\in\bar\cT_h^n}
      \Big\{(r_d,E_{hk}z_2)_{\partial K^n} 
        + (g_d,E_{hk}z_2)_{\partial K^n}\Big\}\\
    &\quad -\sum_{n=1}^N \sum_{j=1}^n\sum_{K\in\bar\cT_h^j}
       \int_{I_n}\Big(\int_{t_{j-1}}^{t_j\wedge t}\cK(t-s)r_d(s)\ ds,
         E_{hk}z_2\Big)_{\partial K}dt.
  \end{split}
\end{equation}

Hence, from \eqref{errorrep:eq2} and space-time cellwise
representation of the other terms in \eqref{errorrep:eq1}, 
we conclude the first  a posteriori error representation
\eqref{errorrep1}.

Now we prove the second error representation \eqref{errorrep2}. 
To this end, in \eqref{errorrep:eq0} we use \eqref{B2} 
the second variant of the bilinear form $B$, in which the convolution 
term has been  rearanged. Therefore we just need to study the rearanged 
convolution term that, similar to \eqref{errorrep:intbyparts1}, is writen as  
\begin{equation}   \label{errorrep:intbyparts2}
  \begin{split}
    \int_0^T\int_t^T&\!\cK(s-t)a(U_1(t),E_{hk}z_2(s))\ ds\ dt\\
    &=\sum_{n=1}^N\int_{I_n} \sum_{K\in\bar\cT_h^n}
       a\Big(U_1(t),\int_t^T\! \cK(s-t)E_{hk}z_2(s)\ ds\Big)_{K}\ dt\\
    &=\sum_{n=1}^N \int_{I_n}  \sum_{K\in \bar\cT_h^n}
      \Big(r_d,\int_t^T\! \cK(s-t)E_{hk}z_2(s)\ ds
      \Big)_{\partial K}\ dt\\
    &\quad +\sum_{n=1}^N\int_{I_n}
        \Big(\sigma_0(U_1)\cdot n,\int_t^T\! \cK(s-t)E_{hk}z_2(s)\ ds
        \Big)_{\Gamma_\tN}\ dt.
  \end{split}
\end{equation}
For the second term of the right side, we exchange the role 
of the variables $s,t$, and we change the order of the time integrals to obtain 
\begin{equation}   \label{errorrep:intbyparts2_2}
  \begin{split}
    \sum_{n=1}^N\int_{I_n}
     &\Big(\sigma_0(U_1)\cdot n,\int_t^T\! \cK(s-t)E_{hk}z_2(s)\ ds
        \Big)_{\Gamma_\tN}\ dt\\
     &=\int_0^T\int_s^T\! 
        \cK(t-s) \big(\sigma_0(U_1(s))\cdot n,
          E_{hk}z_2(t)\big)_{\Gamma_\tN}\ dt\ ds\\
     &=\sum_{n=1}^N\int_{I_n}  
        \Big(\int_0^t\! \cK(t-s) \sigma_0(U_1(s))\cdot n \ ds,E_{hk}z_2(t)
        \Big)_{\Gamma_\tN}\ dt,
  \end{split}
\end{equation}
that, with \eqref{errorrep:intbyparts2}, we have
\begin{equation*}  
  \begin{split}
    \int_0^T\int_t^T\! &\cK(s-t)a(U_1(t),E_{hk}z_2(s))\ ds\ dt\\
    &=\sum_{n=1}^N \int_{I_n}\sum_{K\in \bar\cT_h^n}
      \Big\{\Big(r_d,\int_t^T\! \cK(s-t)E_{hk}z_2(s)\ ds
      \Big)_{\partial K}\\
    &\quad +
        \Big(\int_0^t\! \cK(t-s)\sigma_0(U_1(s))\cdot n\ ds,E_{hk}z_2(t)
        \Big)_{\partial K \cap \Gamma_\tN}\Big\}\ dt.
  \end{split}
\end{equation*}
This, together with \eqref{errorrep:intbyparts1}, imply
\begin{equation}   \label{errorrep:eq3}
  \begin{split}
    \int_0^T\! 
    \Big\{a(U_1(s),&E_{hk}z_2(t)) 
       -\int_t^T\! \!\cK(s-t)a(U_1(t),E_{hk}z_2(s))\ ds
       \Big\}\ dt\\
    &=\sum_{n=1}^N \sum_{K\in\bar\cT_h^n}\!
      \bigg\{\!(r_d,E_{hk}z_2)_{\partial K^n} 
        + (g_d,E_{hk}z_2)_{\partial K^n}\\
    &\quad  
      -\Big(r_d,\int_t^T\!\! \cK(s-t)E_{hk}z_2(s)\ ds\Big)_{\partial K^n}
      \bigg\},
  \end{split}
\end{equation}
that is the replacement for \eqref{errorrep:eq2}. 

Hence, from \eqref{errorrep:eq3} and space-time cellwise
representation of the other terms in \eqref{errorrep:eq1}, 
we conclude the second error representation
\eqref{errorrep2}. 

Finally, we prove the third error representation \eqref{errorrep3}.  
In \eqref{errorrep:eq0} we use \eqref{B2}, the second variant of 
the bilinear form $B$, and similar to \eqref{errorrep:intbyparts1} we have
\begin{equation}   \label{errorrep:intbyparts3}
  \begin{split}
    \int_0^T\int_t^T\! &\cK(s-t)a(U_1(t),E_{hk}z_2(s))\ ds\ dt\\
    &=\sum_{n=1}^N\int_{I_n}\sum_{j=n}^N\int_{t_{j-1}\vee t}^{t_j}
       \cK(s-t)a(U_1(t),E_{hk}z_2(s))\ ds \ dt\\
    &=\sum_{n=1}^N\int_{I_n}\sum_{j=n}^N\int_{t_{j-1}\vee t}^{t_j}
       \cK(s-t)\sum_{K\in \bar\cT_h^j}
       \Big\{
       \Big(r_d(t),E_{hk}z_2(s) \Big)_{\partial K}\\
    &\qquad\qquad\qquad \qquad
       + \Big(\sigma_0(U_1(t))\cdot n,
       E_{hk}z_2(s)\Big)_{\partial K  \cap \Gamma_\tN}
       \Big\}\ ds \ dt.
  \end{split}
\end{equation}
Then, rewriting the second term in the brace 
similar to \eqref{errorrep:intbyparts2_2}, we have
\begin{equation*}   
  \begin{split}
      \int_0^T\int_t^T\! &\cK(s-t)a(U_1(t),E_{hk}z_2(s))\ ds\ dt\\
    &=\sum_{n=1}^N\sum_{j=n}^N\sum_{K\in \bar\cT_h^j}
       \int_{I_n}\int_{t_{j-1}\vee t}^{t_j}
       \cK(s-t)
       \big(r_d(t),E_{hk}z_2(s)\big)_{\partial K}\ ds\ dt\\
    &\qquad 
       +\sum_{n=1}^N\int_{I_n}\sum_{K\in \bar\cT_h^n}
        \Big(\int_0^t \cK(t-s)\sigma_0(U_1(s))\cdot n\ ds,
       E_{hk}z_2(t)\Big)_{\partial K \cap \Gamma_\tN} \ dt
  \end{split}
\end{equation*}
This, together with \eqref{errorrep:intbyparts1}, imply
\begin{equation}   \label{errorrep:eq4}
  \begin{split}
    \int_0^T\! 
    &\Big\{a(U_1(s),E_{hk}z_2(t)) 
       -\int_t^T\! \cK(s-t)a(U_1(t),E_{hk}z_2(s))\ ds
       \Big\}\ dt\\
    &=\sum_{n=1}^N \sum_{K\in\bar\cT_h^n}
      \Big\{(r_d,E_{hk}z_2)_{\partial K^n} 
       + (g_d,E_{hk}z_2)_{\partial K^n}\Big\}\\
    &\qquad 
     -\sum_{n=1}^N\sum_{j=n}^N\sum_{K\in \bar\cT_h^j}
       \int_{I_n}\int_{t_{j-1}\vee t}^{t_j}
       \Big(r_d(t),\cK(s-t)E_{hk}z_2(s)\Big)_{\partial K}\ ds\ dt,
  \end{split}
\end{equation}
that is the replacement for \eqref{errorrep:eq2}. 

Hence, from \eqref{errorrep:eq4} and space-time cellwise
representation of the other terms in \eqref{errorrep:eq1}, 
we conclude the third error representation
\eqref{errorrep3}. 
Now the proof is complete. 
\end{proof}

We note that all error indicators in \eqref{thetas} depend on  
the unknown dual solution $z$, which is not available. 
For strategies to evaluate the error indicators, 
that are based on suitable approximation of the dual solution, 
we refer to  \cite{BangerthGeigerRannacher} 
and \cite{BangerthRannacher:Book}. 
When the spatial meshes change from a time 
level to the next one, evaluation of functions is carried out by means 
of interpolation/extrapolation. We recall that one hanging node is allowed 
per edge or face. 

We also note that the difference between a posteriori error representations 
 \eqref{errorrep1}--\eqref{errorrep3} are the error indicators 
 $\Theta_{5,K}^{n,j}$, $\Theta_{5,K}^n$, and $\Theta_{5,K}^{N,j}$, 
 where we apply the convolution integral either on the residual 
 $r_d$ or on the error term $E_{hk}z_2$. 
 Obviously we choose the cheaper one, depending on the method 
 we choose for computing or estimating $E_{hk}z_2$. 
 It should be noted that the first and the third error representations 
 \eqref{errorrep1} and \eqref{errorrep3} 
 are not space-time cellwise.

\section{{\bf A posteriori error estimates based on global projections}}
In order to evaluate the a posteriori error representations 
\eqref{errorrep1}--\eqref{errorrep3}, 
we need information about the continuous dual solution $z$.
Such information has to be obtained either through a priori
analysis in form of bounds for $z$ in certain Sobolev norms or
through computation by solving the dual problem numerically.
In this context we provide information through a priori analysis
and we leave the investigation on the second case to a later work.

Sometimes the target functional one is interested in is (almost) 
global. Then, instead of numerically approximating the dual solution, 
one can get cheaper error indicators by using analytical a priori estimates for 
the dual solution. We first present a weighted global a posteriori error estimate,
using global $L_2$-projections $\cP_k,\cP_h$ defined in
\eqref{projections}, and error estimates of $\cP_h$ in a
weighted $L_2$-norm.

We recall the weighted global error estimates of the
$L_2$-projection $\cP_h$ \eqref{projections}, see \cite{Boman}.
First we recall some notation.
Let $\cT$ be a given triangulation with mesh function $h$,
and for any simplex $K\in\cT$,
 $\rho_K$ denote the radius of the largest ball contained in
the closure of $K$, that is $\bar K$.
A family $\cF $ of triangulations $\cT$ is called non-degenerate,
if there exsists a constant $c_0$ such that
\begin{equation*}
  c_0=\max_{\cT\in\cF }\max_{K\in\cT} \frac{h_K}{\rho_K}.
\end{equation*}
Let $S_K=\{K'\in\cT : \bar{K'}\cap\bar{K}\neq \varnothing\}$. 
Also $\delta_{\cT}$ and $\delta_{\cF}$ be measures for a given 
triangulation $\cT$ and a given family $\cF$, respectively, defined by
\begin{equation*}   
  \delta_{\cT}=\max_{K\in\cT}\max_{K'\in S_K}
   |1- h_{K'}^2/h_K^2|, \qquad 
   \delta_{\cF }=\max_{\cT\in\cF }\delta_{\cT}.
\end{equation*}

We define the error operators $E_{hk},\,E_h$, and $E_k$ by
\begin{equation*}   
  E_{hk}v=(\cP_k\cP_h-I)v,\quad
  E_hv=(\cP_h-I)v,\quad
  E_kv=(\cP_k-I)v,
\end{equation*}
and we note that
\begin{equation}   \label{erroroperators:rel}
  E_{hk}=E_h+E_k\cP_h.
\end{equation}

We recall standard error estimation for the time projection error 
$E_k$, that is, 
\begin{equation}   \label{errorest-Pk}
  \int_{I_n}|E_k v|\ dt=\int_{I_n}|E_{k,n} v|\ dt 
  \leq C k_n^\gamma \int_{I_n}|\partial_t^\gamma v|\ dt, \quad \gamma=0,1.
\end{equation}
We also quote the error estimates for the spatial projection error 
$E_h$ from \cite{Boman}, which is stated in the following lemma.
\begin{lemma}
Assume that the family $\cF $ of traingulations $\cT$ be  non-degenerate.
Then for sufficiently small $\delta_{\cF }$, there
exists a constant $C$ such that for any triangulation
$\cT\in\cF $ we have, 
\begin{align}  \label{Lemma1:est1}
  \|h^{-s}E_hv\|&\le C\|\nabla^s v\|,
    \quad s=0,1, 2,\ \forall v\in H^s,\\  \label{Lemma1:est2}
  \|h^{-s}\nabla E_hv\|
     &\le C\|\nabla^{s+1} v\|,\quad s=0, 1,\ \forall v\in H^s,
\end{align}
where `$\nabla$' denotes the usual gradient.
\end{lemma}
\noindent
For more details on the practical aspects of $\delta_{\cF }$,
 see \cite{Boman}.

For the next theorem we recall the mesh functions $h_n,\,\bar h_n$
from \eqref{hn}, \eqref{barhn}, and we define the notations
\begin{equation*}
  \begin{aligned}
    &\bar h_{min,n}=\min_{K\in \bar\cT_h^n}\bar h_{K}, 
      &&\bar h_{max,n}=\max_{K\in \bar \cT_h^n}\bar h_{K},\\
    &\cK_{n,T}=\Big( \int_t^T\! \cK(s-t)\ ds \Big)^{1/2}, 
      &&\cK_{n,j}=\Big( \int_{t\vee t_{j-1}}^{t_j}\! \cK(s-t)\ ds \Big)^{1/2}.
  \end{aligned}
\end{equation*}
We also use the scaled trace inequality, for any simplex $K \in \cT_h$,
\begin{equation} \label{scaledtraceineq}
  \|v\|_{\partial K}\leq C(h_K^{-1/2} \|v\|_K 
    + h_K^{1/2} \|\nabla v\|_K).
\end{equation}

\begin{theorem}
Let $u$ be the solution of \eqref{weakformprimary}, and $U$ be the
solution of \eqref{cG} with a non-degenerate family $\cF_h$
of triangulations $\cT_h^n,\,n=0,1,\dots,N$, with sufficiently
small $\delta_{\cF_h}$, such that the weighted global error
estimates \eqref{Lemma1:est1} and \eqref{Lemma1:est2} hold. 
We also assume that $\cK \in L_1(\mathbb{R}^+)$. 
Then, denoting the error $e=U-u$, 
we have the weighted a posteriori error estimate, 
for $\alpha=0,1,2$, $\beta=1,2$, $\gamma=0,1$, 
\begin{equation}   \label{aposteriori:estimate}
  \begin{split}
    |L^*(e)|
    &\le C \max_{0 \leq t\leq T}
     \Big\{\|\nabla^\alpha z_1(t)\| , \|\nabla^\beta z_2(t)\| 
      , \| \partial^{\alpha \wedge 1}_t z_1(t)\| 
      ,\| \partial^\gamma_t z_2(t)\|  
     \Big\} \\
     &\quad \times \Big\{
     \Upsilon_0 + \sum_{n=1}^N \int_{I_n} 
     \Big(\Upsilon_{h}+\Upsilon_{h,\partial K}
       +\Upsilon_k + \Upsilon_{k,\partial K}     
     \Big) \ dt \Big\},
  \end{split}
\end{equation}
where
\begin{equation*}   
  \begin{split}
    &\Upsilon_0 
     =\| h_0^\alpha \big(U_1(0)-u^0\big)\|
       +\| h_0^\beta\big(U_2(0)-v^0\big)\|,\\
    &\Upsilon_{h}
     =\| \bar h_n^\alpha(\dot U_1-U_2)\|+\| \bar h_n^\beta(\dot U_2-f)\|,\\
    &\Upsilon_{h,\partial K}
     =\big(\zeta_{n}(\beta)+\zeta_{n,N}(\beta)\big)
    \Big(\sum_{K\in\bar\cT_h^n}\bar h_{K}^3
       \|r_d\|_{\partial K}^2\Big)^{1/2}\\
    &\qquad\qquad
     +\zeta_{n}(\beta)
      \Big(\sum_{K\in\bar\cT_h^n}\bar h_{K}^3
      \|g_d-g\|_{\partial K}^2\Big)^{1/2},\\
    &\Upsilon_{k}
     =k_n^{\alpha \wedge 1}\|E_k(\dot U_1-U_2)\| 
     +k_n^\gamma \Big(\|E_k\bar A_hU_1\|\\
    &\qquad\qquad 
     +\big\|E_k\int_0^t\!\!
      \cK(t-s)\bar A_hU_1(s)\,ds\big\| +\|E_kf\| \Big),\\
    &\Upsilon_{k,\partial K}
     =k_n^\gamma\Big(\sum_{K\in \bar \cT_h^n}
        h_K^{-1}\|E_kg\|_{\partial K}^2\Big)^{1/2},
  \end{split}
\end{equation*}
with 
$\ \zeta_n(\beta)=\bar h_{min,n}^{\beta-2}, \ 
\zeta_{n,N}(\beta)=\bar h_{min,n}^{-3/2}  \cK_{n,T}
\sum_{j=n}^N \cK_{n,j} \bar h_{max,j}^{\beta-\frac{1}{2}}.$
\end{theorem}
\begin{proof}
Let $z\in\cV^*$ be the solution of the dual problem
\eqref{weakformdual}.
From the definition of the $L_2$ projections $\cP_k,\cP_h$ in
\eqref{projections} and the test space $\cW_{hk}$ in \eqref{discretespaces} we have $\cP_k\cP_hz\in \cW_{hk}$.
Therefore, using \eqref{errorrep:eq0} and 
\eqref{erroroperators:rel} we have,
\begin{equation}   \label{aposteriori:eq1}
  L^*(e)=R(U;E_{hk}z)=R(U;E_h z)+R(U;E_k\cP_h z).
\end{equation}
We study the two terms at the right side of this equation.

For the first term, 
recalling \eqref{errorrep:eq0} and using the a posteriori error representation \eqref{errorrep2}, 
we have
\begin{equation}   \label{R:UEh}
  \begin{split}
   \tR(U;E_h z)&=L^*(e)\\
   &=\sum_{K\in\cT_h^0}\big\{
     \big(U_1(0)-u^0,E_hz_1(0)\big)_K
     +\big(U_2(0)-v^0,E_hz_2(0)\big)_K \big\}\\
   &\quad +\sum_{n=1}^N\int_{I_n}\!\sum_{K\in\bar\cT_h^n}
      \bigg\{
      \big(\dot U_1-U_2,E_hz_1\big)_K
      +(\dot U_2-f,E_hz_2)_K\\
   &\quad\quad   +(r_d,E_hz_2)_{\partial K}
     +(g_d-g,E_hz_2)_{\partial K}\\
    &\quad\quad
     -\Big(r_d,\int_t^T\!
         \cK(s-t)E_hz_2(s)\,ds\Big)_{\partial K}
        \bigg\} \  dt 
    =\sum_{i=1}^2 {\rm{\bf I}}_i 
     +\sum_{i=1}^5 {\rm{\bf II}}_i.
  \end{split}
\end{equation}
We then, for each term, use the Cauchy-Schwarz inequality twice.
First on the local elements $K$ and $\partial K$, to obtain local
$L_2$-norms, and then on the sum over the elements to obtain
global norms such that the weighted global error estimates 
\eqref{Lemma1:est1}--\eqref{Lemma1:est2} can be used.
For ${\rm{\bf I}}_1$, using \eqref{Lemma1:est1}, we have, 
for $\alpha=0,1,2$,
\begin{equation}   \label{I1}
  {\rm{\bf I}}_1\le
   \|h_0^\alpha(U_1(0)-u^0)\|\|h_0^{-\alpha} E_hz_1(0)\|
   \le C \|h_0^\alpha(U_1(0)-u^0)\| \|\nabla^\alpha z_1(0)\|,
\end{equation}
and in a similar way, for $\beta=1,2$,
\begin{equation}   \label{I2}
  \begin{split}
   {\rm{\bf I}}_2
    \le C \|h_0^\beta(U_2(0)-v^0)\| \|\nabla^\beta z_2(0)\|.
  \end{split}
\end{equation}
Now we estimate the second terms ${\rm{\bf II}}_i,\ i=1,\dots,5$. 
Using the Cauchy-Schwarz inequality and the error estimate
\eqref{Lemma1:est1}, we have, for $\alpha=0,1,2$,
\begin{equation}   \label{II1}
  \begin{split}
    {\rm{\bf II}}_1
    &=\sum_{n=1}^N\int_{I_n}\!\sum_{K\in\bar\cT_h^n}
     \big(\dot U_1-U_2,E_hz_1\big)_K \,dt\\
    &\le \sum_{n=1}^N\int_{I_n}\!
     \| \bar h_n^\alpha(\dot U_1-U_2)\| \| \bar h_n^{-\alpha}E_hz_1\|\,dt\\
    &\le C \max_{[0,T]}\|\nabla^\alpha z_1(t)\|
    \sum_{n=1}^N\int_{I_n}\!\| \bar h_n^\alpha(\dot U_1-U_2)\|\,dt,
  \end{split}
\end{equation}
and similarly we obtain, for $\beta=1,2$,
\begin{equation}   \label{II2}
  \begin{split}
    {\rm{\bf II}}_2
    &\le C \max_{[0,T]}\|\nabla^\beta z_2(t)\|
    \sum_{n=1}^N\int_{I_n}\!\| \bar h_n^\beta(\dot U_2-f)\|\,dt.
  \end{split}
\end{equation}
For ${\rm{\bf II}}_3$, we first have,
\begin{equation*}
  \begin{split}
    &{\rm{\bf II}}_3
     \le \sum_{n=1}^N\int_{I_n}\!
     \Big(\sum_{K\in\bar\cT_h^n}
       \bar h_K^3\|r_d(t)\|_{\partial K}^2\Big)^{1/2}
     \Big(\sum_{K\in\bar\cT_h^n}\bar h_K^{-3}
       \|E_hz_2(t)\|_{\partial K}^2\Big)^{1/2} \,dt.
  \end{split}
\end{equation*}
Then, by the scaled trace inequality \eqref{scaledtraceineq} 
and the weighted global error estimates 
\eqref{Lemma1:est1}--\eqref{Lemma1:est2},
we obtain, for $\beta=1,2$,
\begin{equation*}  
  \begin{split}
     \sum_{K\in\bar\cT_h^n}\bar h_K^{-3}
       \|E_hz_2\|_{\partial K}^2
     &\le C \sum_{K\in\bar\cT_h^n}\big\{
       \bar h_K^{-4}\|E_hz_2\|_K^2
       +\bar h_K^{-2}\|\nabla E_hz_2\|_K^2\big\}\\
     &\le C \sum_{K\in\bar\cT_h^n} 
       \bar h_K^{-4+2\beta} 
       \big\{ \| \bar h_K^{-\beta} E_hz_2\|_K^2 
       + \|\bar h_K^{-\beta+1} \nabla E_hz_2\|_K^2 \big\}\\
     &\leq C \bar h_{min,n}^{2(\beta-2)} \| \nabla^\beta z_2 \|^2.
  \end{split}
\end{equation*}
These imply the estimate, for $\beta=1,2$,
\begin{equation}   \label{II3}
  {\rm{\bf II}}_3
   \le C
     \max_{[0,T]}\|\nabla^\beta z_2(t)\|
     \sum_{n=1}^N\int_{I_n}\!\bar h_{min,n}^{\beta-2}
     \Big(\sum_{K\in\bar\cT_h^n}\bar h_K^3
       \|r_d\|_{\partial K}^2\Big)^{1/2}\,dt.
\end{equation}
In a similar way, we have, for $\beta=1,2$,
\begin{equation}   \label{II4}
  {\rm{\bf II}}_4
   \le C
     \max_{[0,T]}\|\nabla^\beta z_2(t)\|
     \sum_{n=1}^N\int_{I_n}\!\bar h_{min,n}^{\beta-2}
     \Big(\sum_{K\in\bar\cT_h^n}\bar h_K^3
       \|g_d-g\|_{\partial K}^2\Big)^{1/2}\,dt.
\end{equation}
Finally we study ${\rm{\bf II}}_5$. 
To this end, first we note that,
\begin{equation*} 
  \begin{split}
    {\rm{\bf II}}_5
     &\le \sum_{n=1}^N\int_{I_n}\!
     \Big(\sum_{K\in\bar\cT_h^n}
      \bar h_K^3\|r_d(t)\|_{\partial K}^2\Big)^{1/2}\\
     &\quad \times
      \Big(\sum_{K\in\bar\cT_h^n}\bar h_K^{-3}
       \Big\|\int_t^T\! \!\cK(s-t)E_hz_2(s)\,ds
       \Big\|_{\partial K}^2\Big)^{1/2} \,dt.
  \end{split}
\end{equation*}
Then, using the Cuachy-Schwarz inequality, 
we have
\begin{equation*}
  \begin{split}
   \sum_{K\in\bar\cT_h^n}\bar h_K^{-3}
       &\Big\|\int_t^T\!\cK(s-t)E_hz_2(s)\,ds
        \Big\|_{\partial K}^2\\
       &\le\sum_{K\in\bar\cT_h^n}\bar h_K^{-3}
        \Big(\int_t^T\!\!\cK(s-t)\|E_hz_2(s)\|_{\partial K}\,ds
        \Big)^2\\
       &\le \sum_{K\in\bar\cT_h^n}\bar h_K^{-3}
        \int_t^T\!\cK(s-t)\,ds
        \int_t^T\!\cK(s-t)\|E_hz_2(s)\|_{\partial K}^2\,ds\\
       &\le \cK_{n,T}^2
        \int_t^T\!\!\cK(s-t)\sum_{K\in\bar\cT_h^n}\bar h_K^{-3}
        \|E_hz_2(s)\|_{\partial K}^2\,ds,
  \end{split}
\end{equation*}
that using the scaled trace inequality  \eqref{scaledtraceineq}  
and \eqref{Lemma1:est1}--\eqref{Lemma1:est2}, we have, for $\beta=1,2$,
\begin{equation*}
  \begin{split}
   \sum_{K\in\bar\cT_h^n}  \bar h_K^{-3}
        \Big\|\int_t^T\! &\cK(s-t)E_hz_2(s)\,ds
        \Big\|_{\partial K}^2\\
       &\le C \bar h_{min,n}^{-3} \cK_{n,T}^2
        \sum_{j=n}^N\int_{t\vee t_{j-1}}^{t_j}\!
        \cK(s-t) \sum_{K\in\bar \cT_h^j}
        \big\{\bar h_K^{-1}
         \| E_hz_2(s)\|_K^2 \\
       &\qquad\qquad\qquad\qquad\qquad\qquad\qquad\qquad
        +\bar h_K \| \nabla E_hz_2(s)\|_K^2\big\}\,ds\\
       &\leq C \bar h_{min,n}^{-3} \cK_{n,T}^2
        \sum_{j=n}^N\int_{t\vee t_{j-1}}^{t_j}\!
        \cK(s-t) \sum_{K\in\bar \cT_h^j}
        \bar h_K^{-1+2\beta}
        \big\{ \| \bar h_K^{-\beta} E_hz_2(s)\|_K^2 \\
       &\qquad\qquad\qquad\qquad\qquad\qquad\qquad\qquad
        +\|  \bar h_K^{-\beta+1} \nabla E_hz_2(s)\|_K^2\big\}\,ds\\
       &\leq C  \bar h_{min,n}^{-3} \cK_{n,T}^2
        \sum_{j=n}^N \bar h_{max,j}^{2(\beta-\frac{1}{2})} \int_{t\vee t_{j-1}}^{t_j}\! 
        \cK(s-t) \| \nabla^\beta z_2(s) \|^2\ ds\\
       &\leq C  \bar h_{min,n}^{-3} \cK_{n,T}^2
        \max_{[t,T]} \| \nabla^\beta z_2(s) \|^2
        \sum_{j=n}^N \bar h_{max,j}^{2(\beta-\frac{1}{2})} \int_{t\vee t_{j-1}}^{t_j}\! 
        \cK(s-t)\ ds.
  \end{split}
\end{equation*}
Hence we have
\begin{equation}   \label{II5}
  \begin{split}
    {\rm{\bf II}}_5
     &\le C \max_{[0,T]}\|\nabla^\beta z_2(t)\|\\
     &\quad\times 
     \sum_{n=1}^N\int_{I_n}\!
     \Big\{      
       \bar h_{min,n}^{-3/2} \cK_{n,T}
       \sum_{j=n}^N \big( \bar h_{max,j}^{\beta-\frac{1}{2}} \cK_{n,j}\big)
       \Big(\sum_{K\in\bar\cT_h^n}\bar h_K^3
       \|r_d\|_{\partial K}^2\Big)^{1/2}
     \Big\}\,dt.
  \end{split}
\end{equation}
Putting \eqref{I1}--\eqref{II5} in \eqref{R:UEh} we conclude, 
for $\alpha=0,1,2$, $\beta=1,2$,
\begin{equation}   \label{R:UEh:estimate}
  \begin{split}
   \tR&(U;E_h z)
   \le C \max_{[0,T]}
    \big\{\|\nabla^\alpha z_1(t)\| , \|\nabla^\beta z_2(t)\| \big\}\\
   &\times\bigg\{
     \|h_0^\alpha \big(U_1(0)-u^0\big)\|
     +\|h_0^\beta \big(U_2(0)-v^0\big)\|\\
   &\quad\quad 
     +\sum_{n=1}^N\int_{I_n}\!\Big\{
     \| \bar h_n^\alpha (\dot U_1-U_2)\|+\| \bar h_n^\beta (\dot U_2-f)\|\\
   &\quad\quad
     +\bar h_{min,n}^{\beta-2}
     \Big(\sum_{K\in\bar\cT_h^n}\bar h_K^3
       \|r_d\|_{\partial K}^2\Big)^{1/2}
     +\bar h_{min,n}^{\beta-2}
     \Big(\sum_{K\in\bar\cT_h^n}\bar h_K^3
       \|g_d-g\|_{\partial K}^2\Big)^{1/2}\\
   &\quad\quad
      +\bar h_{min,n}^{-3/2} \cK_{n,T}
       \sum_{j=n}^N \big( \bar h_{max,j}^{\beta-\frac{1}{2}} \cK_{n,j}\big)
       \Big(\sum_{K\in\bar\cT_h^n}\bar h_K^3
       \|r_d\|_{\partial K}^2\Big)^{1/2}
   \Big\}\,dt\bigg\}.
  \end{split}
\end{equation}

Now we study the second term $R(U;E_kP_hz)$ in \eqref{aposteriori:eq1}. 
We have
\begin{equation*}
  \begin{split}
   \tR(U;E_k\cP_h z)
    &=\tR(U;E_k\cP_h z)\pm \int_0^T\!\big\{
     (\cP_kf,E_k\cP_hz_2)
     +(\cP_kg,E_k\cP_hz_2)_{\Gamma_{\tN}}\big\}\,dt\\
    &=\big(U_1(0)-u^0,E_k\cP_hz_1(0)\big)
     +\big(U_2(0)-v^0,E_k\cP_hz_2(0)\big)\\
    &\quad+\sum_{n=1}^N\int_{I_n}\!
      \Big\{\big(\dot U_1-U_2,E_k\cP_hz_1\big)+a(U_1,E_k\cP_hz_2)\\
    &\quad\quad
      -\int_0^t\!\cK(t-s)
      a\big(U_1(s),E_k\cP_hz_2\big)\,ds+\big(\dot U_2-\cP_kf,E_k\cP_hz_2\big)\\
    &\quad\quad
      -(\cP_kg,E_k\cP_hz_2)_{\Gamma_\tN}
      +(E_kf,E_k\cP_hz_2)
      +(E_kg,E_k\cP_hz_2)_{\Gamma_\tN}\Big\}\,dt.
  \end{split}
\end{equation*}
Recalling the initial data \eqref{initialdata}, 
$U_i(0)=\cP_hu_i(0),\,i=1,2$, 
the first two terms on the right side vanish. 
Besides, from the second equation of \eqref{cG2} we have,
for $V\in V_h^n$,
\begin{equation*}
  \begin{split}
    \int_{I_n}\!\big\{&(\dot U_2,V)
      -(\cP_kf,V)-(\cP_kg,V)_{\Gamma_\tN}\big\}\,dt\\
    &=-\int_{I_n}\!\Big\{a(\cP_kU_1,V)
      -a\Big(\cP_k\int_0^t\!\cK(t-s)U_1(s)\,ds,V\Big)
      \Big\}\,dt.
  \end{split}
\end{equation*}
Hence, we conclude
\begin{equation}   \label{R:UEk}
  \begin{split}
   \tR(U;E_k\cP_h z)
    &=\sum_{n=1}^N\int_{I_n}\!
      \Big\{\big(E_k(\dot U_1-U_2),E_k\cP_hz_1\big)
     -a(E_kU_1,E_k\cP_hz_2)\\
      &\qquad +a\Big(E_k\int_0^t\!
      \cK(t-s)U_1(s)\,ds,E_k\cP_hz_2\Big)\\
    &\qquad
      +(E_kf,E_k\cP_hz_2)+(E_kg,E_k\cP_hz_2)_{\Gamma_\tN}
      \Big\}\,dt.
  \end{split}
\end{equation}
For the last term, 
 we have
\begin{equation*}
  \begin{split}
   \sum_{n=1}^N\int_{I_n}\!
    &(E_kg,E_k\cP_hz_2)_{\Gamma_\tN}\,dt\\
    &\le \sum_{n=1}^N\int_{I_n}\!
     \Big(\sum_{K\in\bar\cT_h^n}
     \bar h_K^{-1}\|E_kg\|_{\partial K}^2\Big)^{1/2}
     \Big(\sum_{K\in\bar\cT_h^n}
     \bar h_K\|E_k\cP_hz_2\|_{\partial K \cap \Gamma_\tN}^2\Big)^{1/2}.
  \end{split}
\end{equation*}
By the scaled trace inequality \eqref{scaledtraceineq} 
and local inverse inequality,  
\begin{equation*}
 h_K\|\nabla \varphi\|_K\leq C\|\varphi\|_K,
  \quad \forall K \in \bar \cT_h^n,\ \forall \varphi \in \bar V_h^n,
\end{equation*}
we have
\begin{equation*}
  \begin{split}
    \sum_{K\in\bar\cT_h^n}
       \bar h_K\|E_k\cP_hz_2\|_{\partial K \cap \Gamma_\tN}^2
    &\le C\sum_{K\in\bar\cT_h^n}
      \big\{\|E_k\cP_hz_2\|_K^2
      + h_K^2\|\nabla E_k\cP_hz_2\|_K^2\big\}\\
    &\le C \sum_{K\in\bar\cT_h^n}
      \big\{\|E_k\cP_hz_2\|_K^2
      +\|E_k\cP_hz_2\|_K^2\big\}
    = C \|E_k\cP_hz_2\|^2.
  \end{split}
\end{equation*}
Hence
\begin{equation*}
  \begin{split}
   \sum_{n=1}^N\int_{I_n}\!
    (E_kg,E_k\cP_hz_2)_{\Gamma_\tN}\,dt
    \le C \sum_{n=1}^N\int_{I_n}\!
    \Big(\sum_{K\in\bar \cT_h^n}
     \bar h_K^{-1}\|E_kg\|_{\partial K}^2\Big)^{1/2}
     \|E_k\cP_hz_2\|.
  \end{split}
\end{equation*}
Considering this in \eqref{R:UEk}, using the Cauchy-Schwarz inequality  
we have
\begin{equation*}  
  \begin{split}
   \tR(U;E_k\cP_h z)
    &\le C\sum_{n=1}^N\int_{I_n}\!
      \bigg\{\|E_k(\dot U_1-U_2)\|\|E_k\cP_hz_1\|
      +\|E_k\bar A_hU_1\|\|E_k\cP_hz_2\|\\
    &\quad 
     +\Big\|E_k\int_0^t\! \cK(t-s)\bar A_hU_1(s)\,ds\Big\|\|E_k\cP_hz_2\|
      +\|E_kf\|\|E_k\cP_hz_2\|\\
     &\quad +\Big(\sum_{K\in\bar\cT_h^n}
     h_K^{-1}\|E_kg\|_{\partial K}^2\Big)^{1/2}
     \|E_k\cP_hz_2\|
      \bigg\}.
  \end{split}
\end{equation*}
From this, 
together with $L_2$-stability of the $L_2$-projection 
$\cP_h$ and the error estimate \eqref{errorest-Pk} for 
$E_k=\cP_k-I$, we conclude, for $\alpha=0,1,2,\ \gamma=0,1$,
\begin{equation}   \label{R:UEk:estimate}
  \begin{split}
   \tR(U;E_k\cP_h z)
    &\le C \max_{[0,T]}
    \big\{\| \partial^{\alpha \wedge 1}_t z_1(t)\| 
      ,\| \partial^\gamma_t z_2(t)\|
    \big\}
      \sum_{n=1}^N\int_{I_n}\!
     \Big\{k_n^{\alpha \wedge 1} \|\dot U_1-U_2\| \\
    &\qquad  +k_n^\gamma\|E_k\bar A_hU_1\|
     +k_n^\gamma \big\|E_k\int_0^t\! \cK(t-s)\bar A_hU_1(s)\,ds\big\| \\
    &\qquad \qquad 
     +k_n^\gamma\|E_kf\| 
     +k_n^\gamma\Big(\sum_{K\in \bar \cT_h^n}
        h_K^{-1}\|E_kg\|_{\partial K}^2\Big)^{1/2}
     \Big\}\,dt.
  \end{split}
\end{equation}

Hence, putting \eqref{R:UEh:estimate} and
\eqref{R:UEk:estimate} in \eqref{aposteriori:eq1}, 
we conclude the a posteriori error estimate
\eqref{aposteriori:estimate}, and this completes the proof.
\end{proof}

We note that, to obtaining a computable error bound in 
\eqref{aposteriori:estimate}, one aspect is to estimate or eliminate 
the norms of the derivative of the dual solutions $z_1,z_2$.  
In practice, exact solutions $z_1,z_2$ are not available. 
Therefore one way is to use accurate numerical approximations 
for $z_1,z_2$, and further the derivatives can be approximated 
by corresponding difference quotients, see \cite{BangerthGeigerRannacher} 
and references therein for practical methods for hyperbolic equations. 
Another possible way is to eliminate these terms by means of stability estimates. 

\begin{example} \label{Example}
For example, let assume that the desired output functional 
$L^*(e)$ is $\|e_1(T)\|$. 
One way, see  \cite{BangerthGeigerRannacher} for other alternatives, 
is to set $j_1=j_2=z_2^T=0$ and $z_1^T=e_1(T)$ in \eqref{aposteriori:estimate}. 
Then, with $\alpha=0,\ \beta=\gamma=1$, and using the stability 
estimate \eqref{stabilitydualineq} with $l=0$, 
we obtain the a posteriori error estimate
\begin{equation*}  
  \begin{split}
    \| e_1(T)\|
    &\le C  \bigg\{
     \|  U_1(0)-u^0\|
     +\|h_0\big(U_2(0)-v^0\big)\|\\
    &\quad +\sum_{n=1}^N\int_{I_n}\!\Big\{
     \| \dot U_1-U_2 \|+\| \bar h_n(\dot U_2-f)\|+\zeta_n(1)
      \Big(\sum_{K\in\bar\cT_h^n}\bar h_K^3
      \|g_d-g\|_{\partial K}^2\Big)^{1/2}\\
    &\quad\quad
    +\big(\zeta_n(1)+\zeta_{n,N}(1)\big)
    \Big(\sum_{K\in\bar\cT_h^n}\bar h_K^3
       \|r_d\|_{\partial K}^2\Big)^{1/2} \\
    &\quad\quad  
    +\| E_k(\dot U_1-U_2) \|
    +k_n\|E_k\bar A_hU_1\|+k_n\big\|E_k\int_0^t\!
      \cK(t-s)\bar A_hU_1(s)\,ds\big\| \\
    &\quad\quad
    +k_n\|E_kf\|
    +k_n\Big(\sum_{K\in\bar \cT_h^n}
      h_K^{-1}\|E_kg\|_{\partial K}^2\Big)^{1/2}
      \Big\}\,dt\bigg\}.
  \end{split}
\end{equation*}
\end{example}

\begin{remark} \label{rem_beta_L2}
We note that when the convolution kenel is slightly more regular 
such that $\|\cK\|_{L_2(\mathbb{R}^+)} < \infty$, following the proof of  
\eqref{II5}, we can replace $\zeta_{n,N}(\beta)$ in \eqref{aposteriori:estimate} 
by
\begin{equation*}
  \zeta_{n,N}(\beta)
   = \bar h_{min,n}^{-3/2} \Big(\int_t^T\cK^2(s-t)\ ds\Big)^{1/2}
    \sum_{j=n}^N k_j^{\frac{1}{2}} \bar h_{max,j}^{\beta-\frac{1}{2}}.
\end{equation*}
\end{remark}

\begin{remark} \label{rem_ZetanN_to_Zetaj}
In the a posteriori estimate \eqref{aposteriori:estimate}, with $\beta=2$, 
we have $\zeta_n(2)=1$, but still $\zeta_{n,N}(2)$ can be big when the spatial 
meshes change. 
We recall that the error representaion \eqref{errorrep2} has been used 
for \eqref{R:UEh}. If  we use the error representaion \eqref{errorrep1}, 
instead, the error indicator 
\begin{equation*}
  \zeta_{n,N}(\beta)
    \Big(\sum_{K\in\bar\cT_h^n}\bar h_K^3\|r_d\|_{\partial K}^2
    \Big)^{1/2},
\end{equation*}
using the notation \eqref{convolution}, is replaced by 
\begin{equation*}
  \sum_{j=1}^n \zeta_j(\beta)
    \Big(\sum_{K\in\bar\cT_h^j}\bar h_K^3\|(\cK*r_d)^j(t)\|_{\partial K}^2
    \Big)^{1/2}.
\end{equation*}
Indeed for ${\rm{\bf II}}_5$, using the Cauchy-Schwarz inequality and 
\eqref{Lemma1:est1}, we can obtain
\begin{equation*}  
  \begin{split}
    {\rm{\bf II}}_5
    &= \sum_{n=1}^N\sum_{j=1}^n\sum_{K \in \bar \cT^j_h}
      \int_{I_n}\Big( \int_{t_{j-1}}^{t_j \wedge t} \cK(t-s)r_d(s)\ ds, 
        E_hz_2 \Big)_{\partial K}\  dt\\
    &\leq \sum_{n=1}^N \int_{I_n} \sum_{j=1}^n\sum_{K \in \bar \cT^j_h}
      \Big\| \int_{t_{j-1}}^{t_j \wedge t} \cK(t-s)r_d(s)\ ds \Big\|_{\partial K} 
        \|E_hz_2\|_{\partial K}\ dt \\
    &\leq \sum_{n=1}^N \int_{I_n} \sum_{j=1}^n
      \Big(\sum_{K \in \bar \cT^j_h} \bar h_K^3
      \| (\cK*r_d)^j(t)\|_{\partial K}^2\Big)^{1/2}
      \Big(\sum_{K \in \bar \cT^j_h} \bar h_K^{-3} 
        \|E_hz_2\|_{\partial K}^2\Big)^{1/2}\ dt\\
    &\leq C \max_{[0,T]} \|\nabla^\beta z_2(t)\|
      \sum_{n=1}^N \int_{I_n} \sum_{j=1}^n \bar h_{min,j}^{\beta-2}
      \Big(\sum_{K \in \bar \cT^j_h} \bar h_K^3
      \| (\cK*r_d)^j(t)\|_{\partial K}^2\Big)^{1/2}\ dt.
  \end{split}
\end{equation*}
Hence with $\beta=2$ we get the optimal order error indicators, 
though this error indicator is not space-time cellwise. 
\end{remark}

\begin{remark} \label{rem_ZetaDelta}
We note that for the error estimate \eqref{aposteriori:estimate} 
there are two types of restriction on the  triangulations;
One by $\zeta_{n,N}$, that measures
the quasiuniformity of the family of triangulation, and the
other by $\delta_{\cF_h}$, that
measure the regularity of the family of triangulations in a
slightly different sense.
Although maybe not explicitly, but $\zeta_{n,N}$ and
$\delta_{\cF_h}$ can be related. In practice we use
finitely many triangulations, that means quasiuniformity holds, 
though possibly with big $\zeta_{n,N}$, 
see Remark \ref{rem_ZetanN_to_Zetaj}. 
This means that we still can use the a posteriori error estimate
\eqref{aposteriori:estimate}.
But when $\delta_{\cF_h}$ is not sufficiently small, the
error estimate \eqref{aposteriori:estimate} does not hold.
This calls for using local interpolants instead of global
$L_2$-projection  $\cP_h$.
In the next section we present an a posteriori error estimate
using interpolation, both in space and time, 
on each space-time cell. 
\end{remark}
\section{{\bf A posteriori error estimates based on local projections}}
We recall the decomposition of the space-time slab $\Omega^n=\Omega\times I_n$ into cells
$K^n=K\times I_n,\,K\in\bar\cT_h^n$.
Let $I_{hk}$ be a standard interpolant, for example linear in space 
and constant in time, such that 
the following error estimates hold for the error operator 
$E_{hk}v=(I_{hk}-I)v$.
\begin{align}   \label{errorIhk:1}
   \|E_{hk}v\|_{K^n}
   &\le C \big(h_K^r\|\nabla^rv\|_{K^n}
     +k_n\| \dot v\|_{K^n}\big),\quad r=1,\,2,\\ \label{errorIhk:2}
   \|\nabla E_{hk}v\|_{K^n}
   &\le C \big(h_K\|\nabla^2v\|_{K^n}
     +k_n\|\nabla \dot v\|_{K^n}\big).
\end{align}

We recall the mesh function $\bar h_n$ from \eqref{barhn}, 
and $\bar h_{max,n}=\max_{K\in \bar\cT_h^n}\bar h_K$.  
We will also use the fact that
\begin{equation}   \label{Omega:n}
  \|v\|_{\Omega^n}^2=\int_{I_n}\!\|v(t)\|^2\,dt
  \le k_n\max_{I_n}\|v(t)\|^2.
\end{equation}


\begin{theorem}
Let $u$ and $U$ be the solutions of \eqref{weakformprimary} 
and \eqref{cG}, respectively, and we assume that $\cK \in L_2(\mathbb{R}^+)$.
Then, with $e=U-u$, we have the weighted a posteriori error estimate, 
for $\alpha=1,2$,
\begin{equation}   \label{aposteriori:estimate2}
  \begin{split}
    | L^*(e)| 
    &\leq C 
     \max_{[0,T]}
     \big\{\|\nabla^{\alpha}z_1(t)\|,\|\nabla^{2}z_2(t)\|,
     \|\dot z_1(t)\|,\|\dot z_2(t)\|,
     \|\nabla\dot z_2(t)\|\big\}\\
    &\quad 
     \times \Big\{\Upsilon_0 + \sum_{n=1}^N (\Upsilon_{n,1}+\Upsilon_{n,2})\Big\}.
  \end{split}
\end{equation}
where
\begin{equation*}  
  \begin{split}
    &\Upsilon_0
    =\|h_0^\alpha \big(U_1(0)-u^0\big)\| + \|h_0^2\big(U_2(0)-v^0\big)\|,\\
   &\Upsilon_{n,1}
    =k_n^{1/2}
     \bigg\{
     \|\bar h_n^\alpha(\dot U_1-U_2)\|_{\Omega^n}
    +k_n\|\dot U_1-U_2\|_{\Omega^n}\\
   &\quad\quad
    +\|\bar h_n^2(\dot U_2-f)\|_{\Omega^n}
     +k_n\|\dot U_2-f\|_{\Omega^n}\\
    &\quad\quad
      +\Big(\sum_{K\in\bar\cT_h^n}\bar h_K^3\|r_d\|_{\partial K^n}^2
    \Big)^{1/2}
     +\Big(\sum_{K\in\bar\cT_h^n}
     \bar h_K^3\|g_h-g\|_{\partial K^n}^2
    \Big)^{1/2}\\
    &\quad\quad
     +k_n 
    \Big(\sum_{K\in\bar\cT_h^n}\bar h_K^{-1}\|r_d\|_{\partial K^n}^2
    \Big)^{1/2}
    +k_n\Big(\sum_{K\in\bar\cT_h^n}\bar h_K^{-1}
    \|g_h-g\|_{\partial K^n}^2
    \Big)^{1/2}\\
    &\quad\quad 
    +k_n
    \Big(\sum_{K\in\bar\cT_h^n}\bar h_K\|r_d\|_{\partial K^n}^2
    \Big)^{1/2}
    +k_n
    \Big(\sum_{K\in\bar\cT_h^n}\bar h_K\|g_h-g\|_{\partial K^n}^2
    \Big)^{1/2}\bigg\},
  \end{split}
\end{equation*}
\begin{equation*}  
  \begin{split}
    &\Upsilon_{n,2}
     =\int_{I_n}\sum_{j=n}^N
      k_j^{1/2}\|\cK\|_{L_2(I_j)} 
       \bigg\{
       \Big(\sum_{K\in\bar\cT_h^j}\bar h_K^3 \|r_d(t)\|_{\partial K}^2\Big)^{1/2} \\
   &\quad\quad\quad    
       +k_j\Big(\sum_{K\in\bar\cT_h^j}\bar h_K^{-1} 
          \|r_d(t)\|_{\partial K}^2\Big)^{1/2}
       +k_j\Big(\sum_{K\in\bar\cT_h^j}
          \bar h_K \|r_d(t)\|_{\partial K}^2\Big)^{1/2}       
       \bigg\}\,dt.
  \end{split}
\end{equation*}
\end{theorem}

\begin{proof}
We write the error representation \eqref{errorrep2} as
\begin{equation}   \label{apost:2:eq0}
  \begin{split}
    L^*(e)
    &=\sum_{K\in\cT_h^0}\Theta_{0,K}
     +\sum_{i=1}^5\sum_{n=1}^N \sum_{K\in\bar\cT_h^n}
      \Theta_{i,K}^n
    ={\rm{\bf I}}_0+\sum_{i=1}^5 {\rm{\bf I}}_i.
  \end{split}
\end{equation}

First we estimate ${\rm{\bf I}}_0$. To this end, recalling $\Theta_{0,K}$
from \eqref{thetas}, we use the Cauchy-Schwarz inequality and the interpolation error estimate \eqref{errorIhk:1} to obtain, for $\alpha=1,2$,
\begin{equation*}
  \begin{split}
    \sum_{K\in\cT_h^0}\big(U_1(0)-u^0,E_{hk}z_1(0)\big)_K
    &\le \sum_{K\in\cT_h^0}\|U_1(0)-u^0\|_K \|E_{hk}z_1(0)\|_K\\
    &\le C\sum_{K\in\cT_h^0}\|U_1(0)-u^0\|_K 
      h_K^\alpha\|\nabla^\alpha z_1(0)\|_K\\
    &\le C\|\nabla^\alpha z_1(0)\|
      \Big(
      \sum_{K\in\cT_h^0}h_K^{2\alpha}\|U_1(0)-u^0\|_K^2
     \Big)^{1/2}\\
    &=C\|\nabla^\alpha z_1(0)\| \| h_0^\alpha \big(U_1(0)-u^0\big)\|.
  \end{split}
\end{equation*}
Similarly we have
\begin{equation*}
  \begin{split}
    \sum_{K\in\cT_h^0}\big(U_2(0)-v^0,E_{hk}z_2(0)\big)_K
    \le C\|\nabla^{2} z_2(0)\| \|h_0^{2}\big(U_2(0)-v^0\big)\|.
  \end{split}
\end{equation*}
From these two estimates we conclude
\begin{equation}   \label{aposteriori2:est1}
  \begin{split}
    {\rm{\bf I}}_0  
    \le C \max \!\big\{\! \|\nabla^\alpha z_1(0)\|,\|\nabla^{2}z_2(0)\| \! \big\}\!
    \big\{\!
     \|h_0^\alpha \big(U_1(0)-u^0\big)\|+\|h_0^{2}\big(U_2(0)-v^0\big)\| \!
    \big\}.
  \end{split}
\end{equation}

For the next term, using the Cauchy-Schwarz inequality and
the error estimate \eqref{errorIhk:1}, we have, for  $\alpha=1,2$,
\begin{equation*}   
  \begin{split}
    {\rm{\bf I}}_1
    &\le \sum_{n=1}^N\sum_{K\in\bar\cT_h^n}
    \|\dot U_1-U_2\|_{K^n}\|E_{hk}z_1\|_{K^n}\\
    &\le C \sum_{n=1}^N\sum_{K\in\bar\cT_h^n}
    \|\dot U_1-U_2\|_{K^n}
    \big(\bar h_K^\alpha\| \nabla^\alpha z_1\|_{K^n}
    +k_n\| \dot z_1\|_{K^n}\big)\\
    &\le C \sum_{n=1}^N
    \Big(\sum_{K\in\bar\cT_h^n}\bar h_K^{2\alpha}
     \|\dot U_1-U_2\|_{K^n}^2\Big)^{1/2}
    \Big(\sum_{K\in\bar\cT_h^n}\|\nabla^\alpha z_1\|_{K^n}^2\Big)^{1/2}\\
    &\quad +C \sum_{n=1}^N k_n
    \Big(\sum_{K\in\bar\cT_h^n}
     \|\dot U_1-U_2\|_{K^n}^2\Big)^{1/2}
    \Big(\sum_{K\in\bar\cT_h^n}
     \| \dot z_1\|_{K^n}^2\Big)^{1/2}\\
    &\leq C \sum_{n=1}^N\Big\{
     \|\bar h_n^\alpha(\dot U_1-U_2)\|_{\Omega^n} 
        \|\nabla^\alpha z_1\|_{\Omega^n}
     +k_n\|\dot U_1-U_2\|_{\Omega^n}
     \|\dot z_1\|_{\Omega^n}\Big\},
  \end{split}
\end{equation*}
that using \eqref{Omega:n} we have
\begin{equation}   \label{aposteriori2:est2}
  \begin{split}
    {\rm{\bf I}}_1
    \le C\max_{[0,T]}\big\{ \|\nabla^\alpha z_1(t)\|, \|\dot z_1(t)\| \big\}\! 
    \sum_{n=1}^N \! k_n^{1/2}\!
     \big\{
     \|\bar h_n^\alpha(\dot U_1-U_2)\|_{\Omega^n}
     +k_n\|\dot U_1-U_2\|_{\Omega^n}
     \big\}.
  \end{split}
\end{equation}
In the same way we obtain
\begin{equation}   \label{aposteriori2:est3}
  \begin{split}
    {\rm{\bf I}}_2
    \le C\max_{[0,T]}\big\{\|\nabla^{2}z_2(t)\|,\|\dot z_2(t)\| \big\}\!
     \sum_{n=1}^N k_n^{1/2}
     \big\{
     \|\bar h_n^{2}(\dot U_2-f)\|_{\Omega^n}
     +k_n\|\dot U_2-f\|_{\Omega^n}
     \big\}.
  \end{split}
\end{equation}
Now for ${\rm{\bf I}}_3$, we use the Cauchy-Schwarz inequality,
the scaled trace inequality \eqref{scaledtraceineq}, 
and the error estimates 
\eqref{errorIhk:1}--\eqref{errorIhk:2} to obtain,
\begin{equation*}   
  \begin{split}
  {\rm{\bf I}}_3
    &\le \sum_{n=1}^N\sum_{K\in\bar\cT_h^n}
      \|g_d-g\|_{\partial K^n}
      \|E_{hk}z_2\|_{\partial K^n}\\
    &\le C\sum_{n=1}^N\sum_{K\in\bar\cT_h^n}
      \|g_d-g\|_{\partial K^n}
      \big\{
     \bar h_K^{-1/2}\|E_{hk}z_2\|_{K^n}+\bar h_K^{1/2}\|\nabla E_{hk}z_2\|_{K^n}
      \big\}\\
    &\le C\sum_{n=1}^N\sum_{K\in\bar\cT_h^n}
      \|g_d-g\|_{\partial K^n}\\
    &\quad \times
      \big\{
      2\bar h_K^{3/2}\|\nabla^{2} z_2\|_{K^n}
      +\bar h_K^{-1/2}k_n\|\dot z_2\|_{K^n}
      +\bar h_K^{1/2}k_n\|\nabla \dot z_2\|_{K^n}
      \big\}\\
    &\le C\sum_{n=1}^N
      \bigg\{
      \Big(\sum_{K\in\bar\cT_h^n}\bar h_K^{3}\|g_d-g\|_{\partial K^n}^2
      \Big)^{1/2} \| \nabla^{2}z_2\|_{\Omega^n}\\
    &\quad\quad\quad 
      +k_n
      \Big(\sum_{K\in\bar\cT_h^n}\bar h_K^{-1}\|g_d-g\|_{\partial K^n}^2
      \Big)^{1/2} \|\dot z_2\|_{\Omega^n}\\
    &\quad\quad\quad 
      +k_n
      \Big(\sum_{K\in\bar\cT_h^n}\bar h_K\|g_d-g\|_{\partial K^n}^2
      \Big)^{1/2} \|\nabla\dot z_2\|_{\Omega^n}\bigg\},
  \end{split}
\end{equation*}
that using \eqref{Omega:n} we have
\begin{equation}   \label{aposteriori2:est4}
  \begin{split}
  {\rm{\bf I}}_3
    &\le C\max_{[0,T]}\big\{\|\nabla^{2}z_2(t)\|,
     \|\dot z_2(t)\|,
     \|\nabla\dot z_2(t)\|\big\}\\
    &\quad\times \sum_{n=1}^N k_n^{1/2}
    \bigg\{
    \Big(\sum_{K\in\bar\cT_h^n}\bar h_K^{3}\|g_d-g\|_{\partial K^n}^2
    \Big)^{1/2}
    +k_n
    \Big(\sum_{K\in\bar\cT_h^n}\bar h_K^{-1}\|g_d-g\|_{\partial K^n}^2
    \Big)^{1/2}\\
    &\quad\quad \quad
    +k_n
    \Big(\sum_{K\in\bar\cT_h^n}\bar h_K\|g_d-g\|_{\partial K^n}^2
    \Big)^{1/2} \bigg\}.
  \end{split}
\end{equation}
And similarly
\begin{equation}   \label{aposteriori2:est5}
  \begin{split}
  {\rm{\bf I}}_4
    &\le C\max_{[0,T]}\big\{\|\nabla^{2}z_2(t)\|,
     \|\dot z_2(t)\|,
     \|\nabla\dot z_2(t)\|\big\}\\
    &\quad\times \sum_{n=1}^N k_n^{1/2}
    \bigg\{
    \Big(\sum_{K\in\bar\cT_h^n}\bar h_K^{3}\|r_d\|_{\partial K^n}^2
    \Big)^{1/2}
    +k_n
    \Big(\sum_{K\in\bar\cT_h^n}\bar h_K^{-1}\|r_d\|_{\partial K^n}^2
    \Big)^{1/2}\\
    &\quad\quad \quad
    +k_n
    \Big(\sum_{K\in\bar\cT_h^n}\bar h_K\|r_d\|_{\partial K^n}^2
    \Big)^{1/2} \bigg\}.
  \end{split}
\end{equation}

Finally we find an estimate for ${\rm{\bf I}}_5$ which includes the convolution terms. 
First, recalling the definition of $\Theta_{5,K}^n$ from \eqref{thetas}, 
we can write ${\rm{\bf I}}_5$ as,
\begin{equation*}   
  \begin{split}
    {\rm{\bf I}}_5
    &=-\sum_{n=1}^N\int_{I_n}\!\int_s^T\! \cK(s-t)\sum_{K\in\bar\cT_h^n}
      \big(r_d(t),E_{hk}z_2(s)\big)_{\partial K}\,ds\,dt \\
    &=-\sum_{n=1}^N\int_{I_n}\sum_{j=n}^N\int_{t\vee t_{j-1}}^{t_j} 
      \cK(s-t)\sum_{K\in\bar\cT_h^j}
      \big(r_d(t),E_{hk}z_2(s)\big)_{\partial K}\,ds\,dt .
  \end{split}
\end{equation*}
Then, using the Cauchy-Schwarz inequality and the assumption 
$\cK \in L_2(\mathbb{R}^+)$, we have
\begin{equation*} 
  \begin{split}
   | {\rm{\bf I}}_5 |
    &\leq \sum_{n=1}^N\int_{I_n}\sum_{j=n}^N\sum_{K\in\bar\cT_h^j}
      \|r_d(t)\|_{\partial K} 
      \int_{t\vee t_{j-1}}^{t_j} \cK(s-t) \|E_{hk}z_2(s)\|_{\partial K}\,ds\,dt \\
    &\leq \sum_{n=1}^N\int_{I_n}\sum_{j=n}^N\sum_{K\in\bar\cT_h^j}
      \|r_d(t)\|_{\partial K}
      \Big(\int_{t\vee t_{j-1}}^{t_j} \cK^2(s-t)\ ds\Big)^{1/2}\\
    &\qquad\qquad\qquad\qquad\qquad\qquad\quad \times 
     \Big( \int_{t\vee t_{j-1}}^{t_j} \|E_{hk}z_2(s)\|_{\partial K}^2\,ds\Big)^{1/2}\,dt \\
    &\leq \sum_{n=1}^N\int_{I_n}\sum_{j=n}^N\sum_{K\in\bar\cT_h^j}
      \|r_d(t)\|_{\partial K} \,dt \ \|\cK\|_{L_2(I_j)}
      \|E_{hk}z_2(s)\|_{\partial K^j},
  \end{split}
\end{equation*}
that,  using the scaled trace inequality \eqref{scaledtraceineq}, 
the error estimates \eqref{errorIhk:1}--\eqref{errorIhk:2}, 
and the Cauchy-Schwarz inequality over the triangles, we have
\begin{equation*}   
  \begin{split}
   | {\rm{\bf I}}_5 |
     &\leq C \sum_{n=1}^N\int_{I_n}\sum_{j=n}^N
      \|\cK\|_{L_2(I_j)}
       \bigg\{
       \Big(\sum_{K\in\bar\cT_h^j}\bar h_K^3 \|r_d(t)\|_{\partial K}^2\Big)^{1/2} 
       \| \nabla^2 z_2(s)\|_{\Omega^j} \\
      &\qquad
       +k_j\Big(\sum_{K\in\bar\cT_h^j}\bar h_K^{-1}\|r_d(t)\|_{\partial K}^2\Big)^{1/2} 
       \| \dot z_2(s)\|_{\Omega^j} \\
      &\qquad 
       +k_j\Big(\sum_{K\in\bar\cT_h^j}\bar h_K \|r_d(t)\|_{\partial K}^2\Big)^{1/2} 
       \| \nabla \dot z_2(s)\|_{\Omega^j} \bigg\}\,dt.
  \end{split}
\end{equation*}
Hence, by \eqref{Omega:n}, we conclude 
\begin{equation}   \label{aposteriori2:est6}
  \begin{split}
   | {\rm{\bf I}}_5 |
   &\leq C \max_{[0,T]}
     \big\{\|\nabla^{2}z_2(t)\|,
     \|\dot z_2(t)\|,
     \|\nabla\dot z_2(t)\|\big\}
     \sum_{n=1}^N\int_{I_n}\sum_{j=n}^N
      \|\cK\|_{L_2(I_j)} k_j^{\frac{1}{2}}\\
   &\qquad\quad\times 
       \bigg\{
       \Big(\sum_{K\in\bar\cT_h^j}\bar h_K^3 \|r_d(t)\|_{\partial K}^2\Big)^{1/2} 
       +k_j\Big(\sum_{K\in\bar\cT_h^j}\bar h_K^{-1} 
          \|r_d(t)\|_{\partial K}^2\Big)^{1/2}\\
   &\qquad\qquad\quad 
       +k_j\Big(\sum_{K\in\bar\cT_h^j}\bar h_K \|r_d(t)\|_{\partial K}^2\Big)^{1/2}       
         \bigg\}\,dt.
  \end{split}
\end{equation}
Putting estimates \eqref{aposteriori2:est1}--\eqref{aposteriori2:est6} in 
\eqref{apost:2:eq0} we conclude the  a posteriori error estimate 
\eqref{aposteriori:estimate2}. Now the proof is complete. 
\end{proof}
\begin{remark}
We note that, for example, we can compute $\|r_d\|_{\partial K^n}$ as
\begin{equation*}
  \begin{split}
     \|r_d\|_{\partial K^n}
    &=\Big(\int_{I_n}\!\big\|\frac{t-t_n}{-k_n}r_d(t_{n-1})
     +\frac{t-t_{n-1}}{k_n}r_d(t_n)\big\|_{\partial K}^2\, dt\Big)^{1/2}\\
    &\le \frac{k_n^{1/2}}{\sqrt{3}}\Big(\|r_d(t_{n-1})
     \|_{\partial K}^2
     +\|r_d({t_n})\|_{\partial K}^2\Big)^{1/2}\\
    &\le \sqrt{\frac{2}{3}}k_n^{1/2}\big(\|r_d(t_{n-1})
     \|_{\partial K}+\|r_d({t_n})\|_{\partial K}\big).
  \end{split}
\end{equation*}
\end{remark}

\begin{remark}
We recall that to prove the a posteriori error estimate 
\eqref{aposteriori:estimate2}, we used the second  
error representaion \eqref{errorrep2} in \eqref{apost:2:eq0}. 
Therefore, the error indicator $\Upsilon_{n,2}$ is not space-time 
cellwise, since the convolution integral applies to the 
error term $E_{hk}z_2$.  Besides, we need the assumption 
$\cK \in L_2(\mathbb{R}^+)$ for the kernel, 
that does not apply to weakly singular 
kernels for which we have $\cK \in L_1(\mathbb{R}^+)$. 
Hence for this case, when $\cK \in L_1(\mathbb{R}^+)$, the error indicator 
$\Upsilon_{n,2}$ is not fulfilled, and we need to use the first 
error representation \eqref{errorrep1} in \eqref{apost:2:eq0}. 
Now we just need to estimate the new ${\rm{\bf I}}_5$, 
that using  the notation \eqref{convolution} is writen in the form 
\begin{equation*}   
  \begin{split}
    {\rm{\bf I}}_5
    &=-\sum_{n=1}^N\int_{I_n}\!
    \sum_{j=1}^n \sum_{K\in\bar\cT_h^j}
    \Big(\int_{t_{j-1}}^{t\wedge t_j} \cK(t-s)r_d(s)\ ds,
      E_{hk}z_2(t)\Big)_{\partial K}\,dt\\
    &=-\sum_{n=1}^N\int_{I_n}\!
    \sum_{j=1}^n\sum_{K\in\bar\cT_h^j}
    \big((\cK*r_d)^j(t),E_{hk}z_2(t)\big)_{\partial K}\,dt.
  \end{split}
\end{equation*}

Then, by the Cauchy-Schwarz inequality and the scaled trace 
inequality \eqref{scaledtraceineq}, we have
\begin{equation*}   
  \begin{split}
    |{\rm{\bf I}}_5|
    &\le C\sum_{n=1}^N\int_{I_n}\sum_{j=1}^n
    \bigg\{
    \Big(\sum_{K\in\bar\cT_h^j}\bar h_K^{-1}
    \|(\cK*r_d)^j(t)\|_{\partial K}^2
    \Big)^{1/2}\\
    &\quad\quad \times
    \Big(\sum_{K\in\bar\cT_h^j}\bar h_K
    \|E_{hk}z_2(t)\|_{\partial K}^2
    \Big)^{1/2} \bigg\}\ dt \\
    &\le C\sum_{n=1}^N\int_{I_n}\sum_{j=1}^n
    \bigg\{
    \Big(\sum_{K\in\bar\cT_h^j}\bar h_K^{-1}
    \|(\cK*r_d)^j(t)\|_{\partial K}^2
    \Big)^{1/2}\\
    &\quad\quad \times
    \Big(\sum_{K\in\bar\cT_h^j}\|E_{hk}z_2(t)\|_K^2
    +\bar h_K^2 \|\nabla E_{hk}z_2(t)\|_K^2
    \Big)^{1/2}\bigg\}\ dt \\
    &\le C\sum_{n=1}^N\int_{I_n}\sum_{j=1}^n
    \bigg\{
    \Big(\sum_{K\in\bar\cT_h^j}\bar h_K^{-1}
    \|(\cK*r_d)^j(t)\|_{\partial K}^2
    \Big)^{1/2}\\
    &\quad\quad \times
    \Big(\|E_{hk}z_2(t)\|+\bar h_{max,j}\|\nabla E_{hk}z_2(t)\|
    \Big) \bigg\}\ dt \\
    &= C\sum_{n=1}^N\int_{I_n}\!
    \|E_{hk}z_2(t)\|\sum_{j=1}^n
    \Big(\sum_{K\in\bar\cT_h^j}\bar h_K^{-1}
    \|(\cK*r_d)^j(t)\|_{\partial K}^2
    \Big)^{1/2}\,dt\\
    &\quad +C\sum_{n=1}^N\int_{I_n}\!
    \|\nabla E_{hk}z_2(t)\|
    \sum_{j=1}^n\bar h_{max,j}
    \Big(\sum_{K\in\bar\cT_h^j}\bar h_K^{-1}
    \|(\cK*r_d)^j(t)\|_{\partial K}^2
    \Big)^{1/2}\,dt,
  \end{split}
\end{equation*}
and using the Cauchy-Schwarz inequality in the integrals over $I_n$, we have 
\begin{equation*}  
  \begin{split}
    |{\rm{\bf I}}_5|\
    &\le C \sum_{n=1}^N \|E_{hk}z_2\|_{\Omega^n}
    \Bigg(\int_{I_n}\!\bigg(
    \sum_{j=1}^n
    \Big(\sum_{K\in\bar\cT_h^j}\bar h_K^{-1}
    \|(\cK*r_d)^j(t)\|_{\partial K}^2
    \Big)^{1/2}\bigg)^2\,dt\Bigg)^{1/2}\\
    &\qquad+C\sum_{n=1}^N \|\nabla E_{hk}z_2\|_{\Omega^n}\\
    &\quad\quad\quad\times
    \Bigg(\int_{I_n}\!\bigg(
    \sum_{j=1}^n \bar h_{max,j}
    \Big(\sum_{K\in\bar\cT_h^j}\bar h_K^{-1}
    \|(\cK*r_d)^j(t)\|_{\partial K}^2
    \Big)^{1/2}\bigg)^2\,dt\Bigg)^{1/2}.
  \end{split}
\end{equation*}

Since by the error estimate \eqref{errorIhk:1} we have
\begin{equation*} 
  \begin{split}
    \int_{I_n}\!\|E_{hk}z_2(t)\|^2\,dt
    &=\int_{I_n}\!\sum_{K\in\bar\cT_h^n}
    \|E_{hk}z_2(t)\|_K^2\,dt
    =\sum_{K\in\bar\cT_h^n}
    \|E_{hk}z_2(t)\|_{K^n}^2\\
    &\le C \sum_{K\in\bar\cT_h^n}
    \big(
    \bar h_K^{4}\|\nabla^{2}z_2\|_{K^n}^2
    +k_n^{2}\| \dot z_2\|_{K^n}^2\big)\\
    &\le C\big(
    \bar h_{max,n}^{4}\|\nabla^{2}z_2\|_{\Omega^n}^2
    +k_n^{2}\|\dot z_2\|_{\Omega^n}^2 \big),
  \end{split}
\end{equation*}
and similarly by \eqref{errorIhk:2} we have
\begin{equation*} 
  \begin{split}
    \int_{I_n}\!\|\nabla E_{hk}z_2(t)\|^2\,dt
    &=\int_{I_n}\!\sum_{K\in\bar\cT_h^n}
    \|\nabla E_{hk}z_2(t)\|_K^2\,dt
    =\sum_{K\in\bar\cT_h^n}
    \|\nabla E_{hk}z_2(t)\|_{K^n}^2\\
    &\le C \sum_{K\in\bar\cT_h^n}
    \big(
    \bar h_K^2\|\nabla^2z_2\|_{\tilde K^n}^2
    +k_n^2\|\nabla \dot z_2\|_{\tilde K^n}^2\big)\\
    &\le C\big(
    \bar h_{max,n}^2\|\nabla^2z_2\|_{\Omega^n}^2
    +k_n^2\|\nabla \dot z_2\|_{\Omega^n}^2 \big),
  \end{split}
\end{equation*}
then, recalling \eqref{Omega:n}, we conclude the estimate
\begin{equation*}  
  \begin{split}
    {\rm{\bf I}}_5
    &\le C \max_{[0,T]}\Big\{\|\nabla^2z_2(t)\|,
     \|\dot z_2(t)\|,
     \|\nabla\dot z_2(t)\|\Big\}
    \Bigg\{
    \sum_{n=1}^N
    k_n^{1/2}\big(\bar h_{max,n}^2+k_n\big)\\
    &\quad\quad\quad\times
    \Bigg(\int_{I_n}\!\bigg(
    \sum_{j=1}^n
    \Big(\sum_{K\in\bar\cT_h^j}\bar h_K^{-1}
    \|(\cK*r_d)^j(t)\|_{\partial K}^2
    \Big)^{1/2}\bigg)^2\,dt\Bigg)^{1/2}\\
    &\quad\quad+C\sum_{n=1}^N
    k_n^{1/2}\big(\bar h_{max,n}+k_n\big)\\
    &\quad\quad\quad\times
    \Bigg(\int_{I_n}\!\bigg(
    \sum_{j=1}^n \bar h_{max,j}
    \Big(\sum_{K\in\bar\cT_h^j}\bar h_K^{-1}
    \|(\cK*r_d)^j(t)\|_{\partial K}^2
    \Big)^{1/2}\bigg)^2\,dt\Bigg)^{1/2} \Bigg\}.
  \end{split}
\end{equation*}
This means that, for the case $\cK \in L_1(\mathbb{R}^+)$, 
the error indicator $\Upsilon_{n,2}$ in the a posteriori error estimate \eqref{aposteriori:estimate2} is
\begin{equation*}   
  \begin{split}
    & k_n^{1/2}
    \Bigg\{
    \big(\bar h_{max,n}^2+k_n\big)\Bigg(\int_{I_n}\!\bigg(
    \sum_{j=1}^n
    \Big(\sum_{K\in\bar\cT_h^j}\bar h_K^{-1}
    \|(\cK*r_d)^j(t)\|_{\partial K}^2
    \Big)^{1/2}\bigg)^2\,dt\Bigg)^{1/2}\\
    &\quad\ \ 
    +\big(\bar h_{max,n}+k_n\big)\Bigg(\int_{I_n}\!\bigg(
    \sum_{j=1}^n \bar h_{max,j}
    \Big(\sum_{K\in\bar\cT_h^j}\bar h_K^{-1}
    \|(\cK*r_d)^j(t)\|_{\partial K}^2
    \Big)^{1/2}\bigg)^2\,dt\Bigg)^{1/2} \Bigg\}.
  \end{split}
\end{equation*}
\end{remark}

We note that the last a posteriori error estimate presented in
\eqref{aposteriori:estimate2}, does not have the restrictions
that were mentioned in Remark \ref{rem_ZetaDelta}.

\section{{\bf Conclusion}}
In this work, a space-time continuous Galerkin finite element method has been 
applied to a hyperbolic type integro-differential equation. 
To provide the main tools for adaptive strategies, mainly based on the DWR 
approach, we have presented three error representations in Theorem 1.  
The main difference between these a posteriori error representations 
is that  we apply the convolution integral 
either on the residual $r_d$ or on the error term $E_{hk}z_2$.  
Obviously we choose the cheaper one, depending on the method 
we choose for computing or estimating $E_{hk}z_2$. 

Evaluation of a posteriori error representations require information about 
the dual solution. One way is to find bounds for the dual solution in certain 
Sobolev norms, that can be estimated, e.g., by means of stability estimates.  
In Theorem 2 we have presented a weighted global a posteriori error 
estimate, that is used when the goal functional is global. The error indicators 
have been separated, such that it can be used as a basis for independent 
adaptation of the time  step and spatial mesh. 
However, for this global error estimate, the  triangulations 
are required to be quasiuniform and also fulfill a certain kind of regularity, 
see Remark \ref{rem_ZetaDelta}. 
This restriction and also when we use local goal functionals call for 
local projections. In Theorem 3 standard local projections have been used 
to obtain a posteriori error estimates. 

For adaptive strategies the error representations can be used to 
accurately estimate the actual errors for getting a stopping criterion of 
the adaptation process. Then we can use the global/local a posteriori 
error estimates for steering the adaptation process. 

Sparse quadrature can be used to overcome  the problem 
with the growing amount of data, in the convolution term, 
that has to be stored and used in time stepping methods. 
However, we note that this is not an issue for exponentially decaying 
memory kernels, in linear viscoelasticity, that are represented as a Prony series.  
This is due to the existence of a recurrence formula for history updating. 
We plan to address adaptivity strategies based on sparse quadrature 
together with numerical adaptation methods such as the DWR approach, 
using the theory presented here, in future work.

\textbf{Acknowledgment.} 
I would like thank Prof. Stig Larsson for helpful discussion and 
constructive comments. 

\bibliographystyle{amsplain}
\bibliography{Thesis}

\end{document}